\providecommand{\pgfsyspdfmark}[3]{}
\documentclass[journal, web]{ieeecolor}

\usepackage{generic}
\usepackage{cite}
\usepackage{amsmath,amssymb,amsfonts}  
\usepackage{graphicx,wrapfig}
\usepackage{amsmath}
\usepackage{amssymb}
\usepackage{hyperref}
\usepackage{mathrsfs}
\usepackage{cite}
\usepackage{ upgreek }
\usepackage{rgsMacros}

\allowdisplaybreaks

\usepackage{hyperref}

\let\labelindent\relax
\usepackage{enumitem}
\usepackage{balance}

 \definecolor{myred}{rgb}{0.86,0.1,0.16} 
 
\newif\ifitsdraft
\def\itsdraft{\global\itsdrafttrue}

\itsdraft  

\newtheorem{exe}{Example}
\newtheorem{prob}{Problem}
\newtheorem{corol}{Corollary}
\newtheorem{ass}{Assumption}
\newtheorem{defin}{Definition}
\newtheorem{cla}{Claim}
\newtheorem{rem}{Remark}
\newtheorem{lem}{Lemma}
\newtheorem{prop}{Proposition}
\newtheorem{thm}{Theorem}
\newtheorem{fct}{Fact}
\newenvironment{lemma}{\begin{lem}}{\hfill $\square$ \end{lem}}
\newenvironment{proposition}{\begin{prop}}{\hfill $\square$ \end{prop}}

\newenvironment{example}{\begin{exe}}{\hfill $\square$ \end{exe}}
\newenvironment{remark}{\begin{rem}}{\hfill $\bullet$ \end{rem}}

\newenvironment{assumption}{\begin{ass}}{\hfill $\bullet$ \end{ass}}
\newenvironment{theorem}{\begin{thm}}{\hfill $\square$ \end{thm}}
\newenvironment{definition}{\begin{defin}}{\hfill 
$\bullet$ \end{defin}}
\newenvironment{claim}{\begin{cla}}{\hfill $\bullet$ \end{cla}}

\begin{document}

\title{\LARGE \bf On a Strong Robust-Safety Notion for Differential Inclusions}

\author{Mohamed Maghenem \quad  Diana Karaki 
\thanks{ The authors are with University of Grenoble Alpes,  CNRS,  Grenoble INP,  France.  Email : mohamed.maghenem@gipsa-lab.fr.} 
}

\maketitle

\begin{abstract}
\blue{A dynamical system is strongly robustly safe provided that it remains safe in the presence of a continuous and positive perturbation, named robustness margin,  added to both the argument and the image of the right-hand side (the dynamics).   Therefore,  in comparison with existing robust-safety notions,  where the continuous and positive perturbation is added only to the  image of the right-hand side,  the proposed notion is shown to be relatively stronger in the context of set-valued right-hand sides.   Furthermore,  we distinguish between strong robust safety and \textit{uniform strong robust safety},  which requires the existence of a constant robustness margin.   The first part of the paper proposes sufficient conditions for strong robust safety in terms of barrier functions.   The proposed conditions  involve only the barrier function and the system's right-hand side.   Furthermore,  we establish the equivalence between strong robust safety and the existence of a smooth barrier certificate.  The second part of the paper proposes scenarios, under which,  strong robust safety implies uniform strong robust safety.  Finally,  we propose  sufficient conditions for the latter notion  in terms of barrier functions. }
\end{abstract}   
            
\section{introduction}    

Safety is one of the most important properties to ensure for a dynamical system, as it requires  the solutions starting from a given set of initial conditions to never reach a given unsafe region  \cite{prajna2007framework}.  
Depending on the considered application, reaching the unsafe set can correspond to the impossibility of applying a predefined feedback law due to saturation or loss of fidelity in the model,  or simply,  colliding with an obstacle or another system.  Ensuring safety is in fact a key step in many engineering  applications such as traffic regulation \cite{ersal2020connected},  aerospace \cite{9656550},  and human-robot interactions \cite{9788028}. 

\subsection{Motivation}

In this work, we consider the case where the dynamical system is given by the differential inclusion 
\begin{align} \label{eq1}
\Sigma : \quad  \dot x \in F(x) \qquad x \in \mathbb{R}^n. 
\end{align}
Furthermore,  we introduce two subsets $(X_o,X_u) \subset \mathbb{R}^n \times \mathbb{R}^n$,  where $X_o$ represents the subset of initial conditions and $X_u$ represents the unsafe set.  Hence, we necessarily have $X_o \cap X_u = \emptyset$. In the following we recall the definition of safety as in classical literature \cite{prajna2004safety, prajna2007framework}. 
\begin{definition} [Safety] 
System $\Sigma$ is safe with respect to $(X_o, X_u)$ if, for each solution $\phi$ starting from  $x_o \in X_o$, we have $\phi(t) \in \mathbb{R}^n \backslash X_u$ for all $t \in \dom \phi$.
\end{definition}
It is worthnoting that we usually show safety by showing the existence of a subset $K \subset \mathbb{R}^n$, with $X_o \subset K$ and $K \cap X_u = \emptyset$, that is  \textit{forward invariant}; namely,  the solutions to $\Sigma$ starting from $K$ remain in $K$ \cite{taly2009deductive}. 

Since the safety notion is not robust in nature,  inspired by the works in \cite{prajna2007framework,ratschan2018converse,wisniewski2016converse}, the following robust and uniform robust safety notions are introduced  in \cite{RubSafPI} and \cite{RubSafPII}.  
\begin{definition} [Robust safety] \label{def-RegRobSafe} 
System $\Sigma$ is robustly safe with respect to $(X_o, X_u)$ if there exists a continuous function $\varepsilon : \mathbb{R}^n \rightarrow \mathbb{R}_{>0}$ such that the system $\Sigma_{\varepsilon}$
given by 
\begin{align} \label{eq2}
\Sigma_\varepsilon : ~ \dot{x} \in F(x) + \varepsilon(x) \mathbb{B}  \qquad  x \in \mathbb{R}^n,
\end{align}
where \blue{$\mathbb{B}$ is the closed unite ball centered at the origin,  whose dimension is understood from the context},
is safe with respect to $(X_o,X_u)$. 
Such a function $\varepsilon$ is called 
\textit{robust-safety margin}. 
Furthermore,  $\Sigma$ is uniformly robustly safe with respect to $(X_o, X_u)$ if it admits a constant robust-safety margin. 
\end{definition}
By definition,  uniform robust safety implies robust safety,  which in turn,  implies safety. The opposite directions are not always true; we refer the reader to  \cite[Examples 1 and 2]{9683684}.

\subsection{Background}

To analyze safety and robust safety without computing solutions, barrier functions have been used intensively in recent years,  and different types of conditions \textcolor{blue}{are derived; see    \cite{ratschan2018converse, wisniewski2016converse,  liu2020converse, JANKOVIC2018359,seiler2021control,  9636670}}.  The barrier functions of interest in this work are scalar functions with distinct signs on $X_o$ and $X_u$.  Any scalar function verifying such a property is called \textit{barrier function candidate}. 
\begin{definition}[Barrier function candidate]
A function $B : \mathbb{R}^n \rightarrow \mathbb{R}$ is a barrier function candidate with respect to $(X_o,X_u)$ if
\begin{equation}
\label{eq.3}
\begin{aligned} 
 B(x) > 0 \quad \forall x \in X_u \quad \text{and} \quad 
 B(x) \leq 0 \quad \forall x \in  X_o.  
\end{aligned}
\end{equation}
\end{definition}
A barrier function candidate $B$ defines the zero-sublevel set
\begin{align} \label{eq.4} 
K := \left\{x \in \mathbb{R}^n : B(x) \leq 0 \right\}.
\end{align} 
Note that the set $K$ is closed provided that $B$ is at least lower semicontinuous,  $X_o \subset K$,  and $K \cap X_u = \emptyset$.  As a result, we can show safety and robust safety by, respectively, showing:
\begin{itemize}
\item Forward invariance of the set $K$ for $\Sigma$.
\item Forward invariance of the set $K$ for $\Sigma_\varepsilon$. 
\end{itemize}
When $B$ is at least continuous, the set $K$ is forward invariant for $\Sigma$ if \cite{9705088}
\begin{enumerate} 
[label={($\star$)},leftmargin=*]
\item \label{item:star1} 
\blue{There exists $U(K)$,  an open neighborhood of the set $K$,  such that,  along each solution $\phi$ to $\Sigma$ starting from $U(K) \backslash K$ and remaining in $U(K) \backslash K$\footnote{Note that $\phi$ is not necessarily a maximal solution,  it can be any solution whose range  is subset  of $U(K) \backslash K$.},   the map $t \mapsto B(\phi(t))$ is non-increasing.}
\end{enumerate}
The monotonicity condition \ref{item:star1}  can be characterized using infinitesimal conditions depending on the regularity of $B$.  In particular, when $B$ is continuously differentiable, \ref{item:star1} holds 
provided that     
\begin{align} \label{eqsafe}
\langle \nabla B (x), \eta \rangle \leq 0 \qquad \forall (\eta,x) \in (F(x), U(K)\backslash K).
\end{align} 
Note that condition \eqref{eqsafe} certifies robust safety if we replace $F(x)$ therein  by $F(x) + \varepsilon(x) \mathbb{B}$,   for a continuous and positive function $\varepsilon$ to be found.  Nonetheless,  it is more interesting to characterize robust safety without searching for a robustness margin, i.e., using a condition that does not involve $\varepsilon$.  As a result,  the following condition is used in \cite{ratschan2018converse, wisniewski2016converse,9444774}. 
\begin{align} \label{eqrsafe}
\langle \nabla B (x), \eta \rangle < 0 \qquad \forall (\eta,x) \in (F(x), 
\partial K).
\end{align} 
See also \cite{RubSafPI} for other possible characterizations of robust and 
uniform robust safety.   

Note that many works use numerical tools  such as SOS and learning-based algorithms to certify safety by numerically searching an appropriate barrier certificate; see \cite{DAI201762,10.1007/978-3-642-39799-8_17,9303785}.  Hence, it is natural to raise the following questions:
\\
\textbf{Converse safety:} Given a safe system $\Sigma$,    
\begin{itemize}
\item is there a continuous barrier candidate $B$ satisfying \ref{item:star1}? 
\item is there a continuously differentiable barrier function candidate $B$ satisfying \eqref{eqsafe}?
\end{itemize}
\textbf{Converse robust safety:} Given a robustly-safe system $\Sigma$, is there a continuously differentiable barrier function candidate $B$ satisfying \eqref{eqrsafe}?

The converse safety problem has been studied in  \cite{prajna2005necessity, konda2019characterizing,9705088},  where the answer is shown to be negative unless extra assumptions on the solutions or on the sets $(X_o, X_u)$ are made,  or time-varying barrier functions are used;  see the discussion in the latter reference. 
Whereas,  the converse robust-safety problem is studied in \cite{wisniewski2016converse}, \cite{ratschan2018converse}, \cite{9444774}, and \cite{RubSafPII}. To the best of our knowledge,  the latter reference proposed the following least restrictive assumptions for the answer to be positive.
\begin{assumption} \label{ass2} 
The map $F$ is continuous with nonempty,  convex,  and compact images,  and $\cl(X_o) \cap \cl (X_u) = \emptyset$. 
\end{assumption}
\blue{As an application,   this converse result is used in  \cite{RubSafPII} to show the existence of a self-triggered control strategy,  with inter-event times larger than a constant,  such that  the self-triggered closed-loop system is safe,  provided that  the continuous implementation of the controller guarantees robust safety.  }

\subsection{Contributions}  
 
In this paper, we introduce a new robust-safety notion, called strong robust safety.  This notion is stronger than the robust-safety notion used in existing literature. 
In particular,  strong robust safety for $\Sigma$ corresponds to safety for a perturbed version of $\Sigma$, denoted $\Sigma^s_\varepsilon$.  
Different from $\Sigma_\varepsilon$, the perturbation term $\varepsilon$ in $\Sigma^s_\varepsilon$ is added both to the argument and the image of the map $F$; thus,  the right-hand side of $\Sigma^s_\varepsilon$ is \textcolor{blue}{$\co \{ F(x + \varepsilon(x) \mathbb{B}) \} + \varepsilon(x) \mathbb{B}$, where $\co$ denotes the convex hull, which we use to avoid well-posedness issues}.  Furthermore,  when $\Sigma^s_\varepsilon$ is safe with $\varepsilon$ constant,  we recover the uniform strong robust-safety notion.   \blue{Note that robustness with respect to sufficiently-small perturbations affecting both the argument and the image of the dynamics is well studied in the context of asymptotic stability; see  \cite{COCV_2000__5__313_0} and \cite{goebel2012hybrid}.  Furthermore,   several applications,  where guaranteeing robustness with respect to sufficiently small perturbations is useful,  are presented in  \cite{RubSafPI}.  Those example are reminiscent from the contexts of 
self-triggered control, singularly perturbed systems,  and assume-guarantee contracts.  }

A motivation to study strong robust safety comes from the context of control loops subject to measurement and actuation noises. That is,  without loss of generality,  given a control system of the form 
$$ \Sigma_u : ~ \dot{x} = u \qquad x \in \mathbb{R}^n $$ 
and a feedback law $u \in F(x)$.  We can see that the nominal closed-loop system has the form of $\Sigma$.  As shown in Figure \ref{fig1},  having the nominal closed-loop system strongly robustly safe implies that the control loop tolerates small perturbations in both sensing and actuation.  Whereas,  classical robust safety induces that the control loop tolerates perturbations in actuation only.
\begin{figure}[h]
    \centering
    \includegraphics[width=0.4\textwidth]{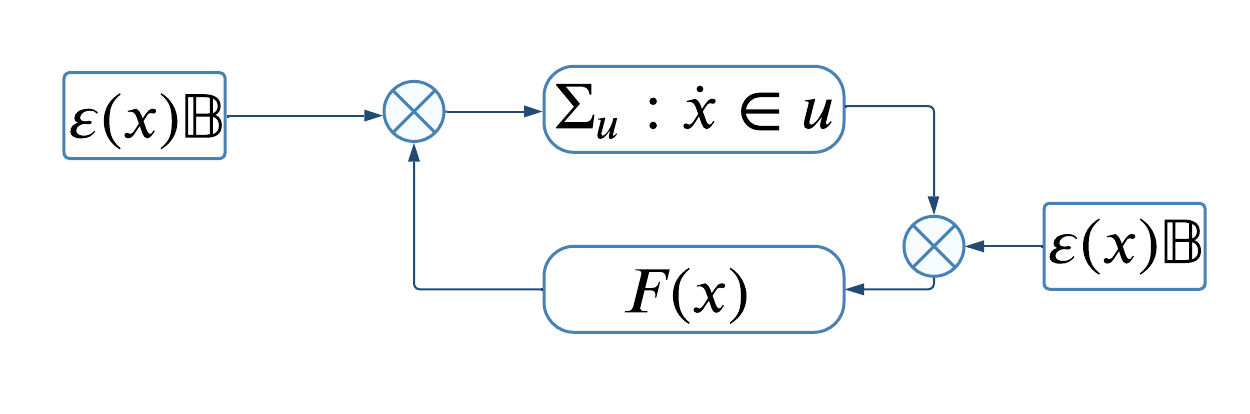}
    \caption{A control loop subject to measurement and actuation noises}
    \label{fig1}
\end{figure}
Clearly,  uniform strong robust safety implies strong robust safety, which implies robust safety. However, as we show via  counterexamples, the opposite directions are not always true.

  Following the spirit of \cite{RubSafPI} and \cite{RubSafPII},  the main contributions of the paper can be divided into two parts.   \blue{In the first part,   we show that, under mild regularity assumptions on $F$, the sufficient conditions for robust safety in \cite[Theorem 1]{RubSafPI} are strong enough to guarantee strong robust safety.   Furthermore,  we are able to establish the equivalence between strong robust safety and the existence of a barrier functions candidate $B$ satisfying \eqref{eqrsafe} under very mild assumptions on 
$\Sigma$ and the sets $(X_o,X_u)$.  Different from the converse robust-safety theorem in \cite{RubSafPII},   the proposed converse strong robust safety theorem does not require $F$ to be continuous.  
In the second part of the paper,  we study uniform strong robust safety.  That is,  as we show via a counterexample,   the sufficient conditions for uniform robust safety in \cite{RubSafPI}  are not strong enough to guarantee uniform strong robust safety.  As a result,  a set of sufficient conditions for uniform strong robust safety is derived based on the smoothness of $B$ and $F$,  and depending on whether $X_o$ and $X_u$ are bounded or not.  Those sufficient conditions are more restrictive than those certifying robust safety and uniform robust safety,   which is natural as we want to ensure a relatively stronger property.}

\textcolor{blue}{The rest of the paper is organized as follows.  Preliminaries are in Section \ref{sec.2}.   The problem formulation is in Section \ref{sec.03}.  
The characterization of strong robust safety is in Section \ref{sec.4}.  Finally,  sufficient conditions for uniform  strong robust safety are  in Section \ref{sec.5}. }

\textbf{Notation.}  For $x$, $y \in \mathbb{R}^n$, $x^{\top}$ denotes the transpose of $x$, $|x|$ the Euclidean norm of $x$,  and $\langle x, y \rangle:= x^\top y $ the inner product between $x$ and $y$.  For a set $K \subset \mathbb{R}^n$, we use $\mbox{int}(K)$ to denote its interior, $\partial K$ its boundary,  $U(K)$ any open neighborhood around $K$,  \blue{and  $|x|_K := \inf \{ |x-y| : y \in K \}$ to denote the distance between $x$ and the set $K$}.   For $O \subset \mathbb{R}^n$,  $K \backslash O$ denotes the subset of elements of $K$ that are not in $O$ and $|O-K|_H$ denotes the Hausdorff distance between $O$ and $K$.     For a function $\phi : \mathbb{R}^n \rightarrow \mathbb{R}^m$, $\dom \phi$ denotes the domain of definition of $\phi$.  By $F : \mathbb{R}^m \rightrightarrows \mathbb{R}^n $, we denote a set-valued map associating each element $x \in \mathbb{R}^m$ into a subset $F(x) \subset \mathbb{R}^n$.  For a set $D \subset \mathbb{R}^m$, $F(D) := \{ \eta \in F(x) : x \in D \}$.  By $\graph (F) \subset \mathbb{R}^{mn}$,  we denote the graph of $F$.   For a differentiable map $x \mapsto B(x) \in \mathbb{R}$,  $\nabla_{x_i} B$ denotes the gradient of $B$ with respect to $x_i$,  $i \in \{1,2,...,n\}$, and $\nabla B$ denotes the gradient of $B$ with respect to $x$.   

\section{Preliminaries} \label{sec.2}

\subsection{Set-Valued vs Single-Valued Maps} \label{sec.1a}

Let a set-valued map $F: K \rightrightarrows \mathbb{R}^n$, where $K \subset \mathbb{R}^m$. 

\begin{itemize}
\item The map $F$ is \textit{outer semicontinuous} at $x \in K$ if,  for every sequence $\left\{x_i\right\}^{\infty}_{i=0} \subset K$ and for every sequence  $\left\{ y_i \right\}^{\infty}_{i=0} \subset \mathbb{R}^n$ with $\lim_{i \rightarrow \infty} x_i = x$, $\lim_{i \rightarrow \infty} y_i = y \in \mathbb{R}^n$, and $y_i \in F(x_i)$ for all $i \in \mathbb{N}$, we have $y \in F(x)$;  see \cite{rockafellar2009variational}. 

\item 
\blue{ The map $F$ is  \textit{lower semicontinuous}  at $x \in K$ if,  for each $y \in F(x)$ and for each sequence  $\left\{x_i\right\}^{\infty}_{i=0} \subset K$ converging to $x$, there exists a sequence $\left\{y_i\right\}^{\infty}_{i=0}$,  with $y_i \in F(x_i)$ for all $i \in \mathbb{N}$, that converges to $y$; see  \cite[Definition 1.4.2]{aubin2009set}.  }

\item The map $F$ is \textit{upper semicontinuous} at $x \in K$ if,  for each $\varepsilon > 0$,  there exists a neighborhood of $x$, denoted $U(x)$,  such that for each $y \in U(x) \cap K$, $F(y) \subset F(x) + \varepsilon \mathbb{B}$; see \cite[Definition 1.4.1]{aubin2009set}.
\item The map $F$ is said to be \textit{continuous} at $x \in K$ if it is both upper and lower semicontinuous at $x$.

\item The map $F$ is \textit{locally bounded} at $x \in K$,  if there exists a neighborhood of $x$, denoted $U(x)$, and $\beta > 0$ such that  $|\zeta| \leq \beta$ for all $\zeta \in F(y)$ and for all $y \in U(x) \cap K$.  
\end{itemize}

Furthermore,  the map $F$ is upper, lower, outer semicontinuous,  continuous, or locally bounded if, respectively,  so it is for all $x \in K$. 

Let a single-valued map $B: K \rightarrow \mathbb{R}$,  where $K \subset \mathbb{R}^m$. 

\begin{itemize}
\item  $B$ is \textit{lower semicontinuous} at $x \in K$ if, for every sequence $\left\{ x_i \right\}_{i=0}^{\infty} \subset K$ such that $\lim_{i \rightarrow \infty} x_i = x$, we have $\liminf_{i \rightarrow \infty} B(x_i) \geq B(x)$. 
\item  $B$ is \textit{upper semicontinuous} at $x \in K$ if, for every sequence $\left\{ x_i \right\}_{i=0}^{\infty} \subset K$ such that $\lim_{i \rightarrow \infty} x_i = x$, we have $\limsup_{i \rightarrow \infty} B(x_i) \leq B(x)$.  
\item  $B$ is \textit{continuous} at $x \in K$ if it is both upper and lower semicontinuous at $x$. 
\end{itemize}

Furthermore, $B$ is upper, lower semicontinuous, or continuous if, respectively, so it is for all $x \in K$.

\subsection{Differential 
Inclusions: Concept of Solutions and Well-Posedness} 
\label{sec.03}

Consider the differential inclusion $\Sigma$ in \eqref{eq1}, with the right-hand side $F:\mathbb{R}^n \rightrightarrows \mathbb{R}^n$ being a set-valued map
 \cite{aubin2012differential}.

\begin{definition} [Concept of solutions] \label{def.ConSol}
 $\phi : \dom \phi \rightarrow \mathbb{R}^n$, where $\dom \phi$ is of the form $[0,T]$ or $[0,T)$ for some $T \in \mathbb{R}_{\geq 0} \cup \{+ \infty\}$, is a solution to $\Sigma$ starting from $x_o \in \mathbb{R}^n$ if $\phi(0) = x_o$,  the map $t \mapsto \phi(t)$ is locally absolutely continuous,  and
$\dot{\phi}(t) \in F(\phi(t))$ for almost all $t \in \dom \phi$.
\end{definition}

 A solution $\phi$ to $\Sigma$ starting from $x_o \in \mathbb{R}^n$ is forward complete if $\dom \phi$ is unbounded. It is maximal if there is no solution $\psi$ to $\Sigma$ starting from $x_o$ such that $\psi(t) = \phi(t)$ for all $t \in \dom \phi$ and $\dom \phi$ strictly included in $\dom \psi$. 

\begin{remark}
We say that $\phi : \dom \phi \rightarrow \mathbb{R}^n$ is a backward solution to $\Sigma$ if $\dom \phi$ is of the form $[-T,0]$ or $(-T,0]$ for some $T \in \mathbb{R}_{\geq 0} \cup \{+ \infty\}$ and $\phi(-(\cdot)) : - \dom \phi \rightarrow \mathbb{R}^n$ is \blue{a solution to $\Sigma^- : \dot{x} \in - F(x)$}. 
 \end{remark}

The differential inclusion $\Sigma$ is said to be well posed if the right-hand side $F$ satisfies the following assumption.

\begin{assumption} \label{ass1} 
The map $F: \mathbb{R}^n \rightrightarrows \mathbb{R}^n $ is upper semicontinuous and $F(x)$ is nonempty, compact, and convex for all $x \in \mathbb{R}^n$.
\end{assumption}

Assumption \ref{ass1} guarantees the existence of solutions from any $x \in \mathbb{R}^n$,  as well as adequate structural properties for the set of solutions to $\Sigma$;  see \cite{aubin2012differential, Aubin:1991:VT:120830}.  Furthermore, when $F$ is single valued, Assumption \ref{ass1} reduces to just continuity of $F$.

\begin{remark}
We recall that having $F$ upper semicontinuous with compact images is equivalent to having  $F$ is outer semicontinuous and  locally bounded; see \cite[Theorem 5.19]{rockafellar2009variational} and  \cite[Lemma 5.15]{goebel2012hybrid}.
\end{remark}

\section{Problem Formulation} \label{sec.3}

Given the differential inclusion $\Sigma$, we introduce its ``strongly'' perturbed  version $\Sigma^s_\varepsilon$ given by 
\textcolor{blue}{\begin{align} \label{eq3}
\Sigma^s_\varepsilon : ~ \dot{x} \in \co \{ F(x + \varepsilon(x) \mathbb{B}) \} + \varepsilon(x) \mathbb{B}  \qquad  x \in \mathbb{R}^n. 
\end{align}}
Next, we introduce the proposed strong robust-safety notion.
\begin{definition} [Strong robust safety] \label{def-RobSafe0} 
System $\Sigma$ is said to be strongly robustly safe with respect to $(X_o, X_u)$ if there exists a continuous function $\varepsilon : \mathbb{R}^n \rightarrow \mathbb{R}_{>0}$ such that $\Sigma^s_{\varepsilon}$ is safe with respect to $(X_o,X_u)$.  Such a function $\varepsilon$ is called a \textit{strong robust-safety margin}. 
\end{definition} 
We also introduce the uniform strong robust-safety notion. 

\begin{definition} [Uniform strong robust safety] \label{def-RobSafe1}
System $\Sigma$ is said to be uniformly strongly robustly safe with respect to $(X_o, X_u)$ if it is strongly robustly safe and admits a constant strong robust-safety margin.
\end{definition}

Clearly,  strong robust safety implies robust safety. However,  as we show in the following example, the opposite is not always true.  

\begin{example} \label{exp1}
Consider the system $\Sigma$ with $n = 1$ and 
\begin{equation*}
\begin{aligned}
F(x) :=
\begin{cases}
2 & \text{if} \qquad x \leq 0 \\
[-1,2] & \text{if} \qquad x=0 \\
-1  & \text{if} \qquad x \geq 0.
\end{cases}
\end{aligned}
\end{equation*}
Consider the sets (see Fig.  \ref{FigFigo1})
\begin{align*}
X_o := \{x \in \mathbb{R} : x \in [-2,0] \} \quad \text{and} \quad 
X_u :=  \mathbb{R} \backslash X_o.
\end{align*}

\begin{figure}
\begin{center}
  \def\svgwidth{0.7\columnwidth}
\begingroup%
  \makeatletter%
  \providecommand\color[2][]{%
    \errmessage{(Inkscape) Color is used for the text in Inkscape, but the package 'color.sty' is not loaded}%
    \renewcommand\color[2][]{}%
  }%
  \providecommand\transparent[1]{%
    \errmessage{(Inkscape) Transparency is used (non-zero) for the text in Inkscape, but the package 'transparent.sty' is not loaded}%
    \renewcommand\transparent[1]{}%
  }%
  \providecommand\rotatebox[2]{#2}%
  \newcommand*\fsize{\dimexpr\f@size pt\relax}%
  \newcommand*\lineheight[1]{\fontsize{\fsize}{#1\fsize}\selectfont}%
  \ifx\svgwidth\undefined%
    \setlength{\unitlength}{268.1918283bp}%
    \ifx\svgscale\undefined%
      \relax%
    \else%
      \setlength{\unitlength}{\unitlength * \real{\svgscale}}%
    \fi%
  \else%
    \setlength{\unitlength}{\svgwidth}%
  \fi%
  \global\let\svgwidth\undefined%
  \global\let\svgscale\undefined%
  \makeatother%
  \begin{picture}(1,0.74439631)%
    \lineheight{1}%
    \setlength\tabcolsep{0pt}%
    \put(0,0){\includegraphics[width=\unitlength,page=1]{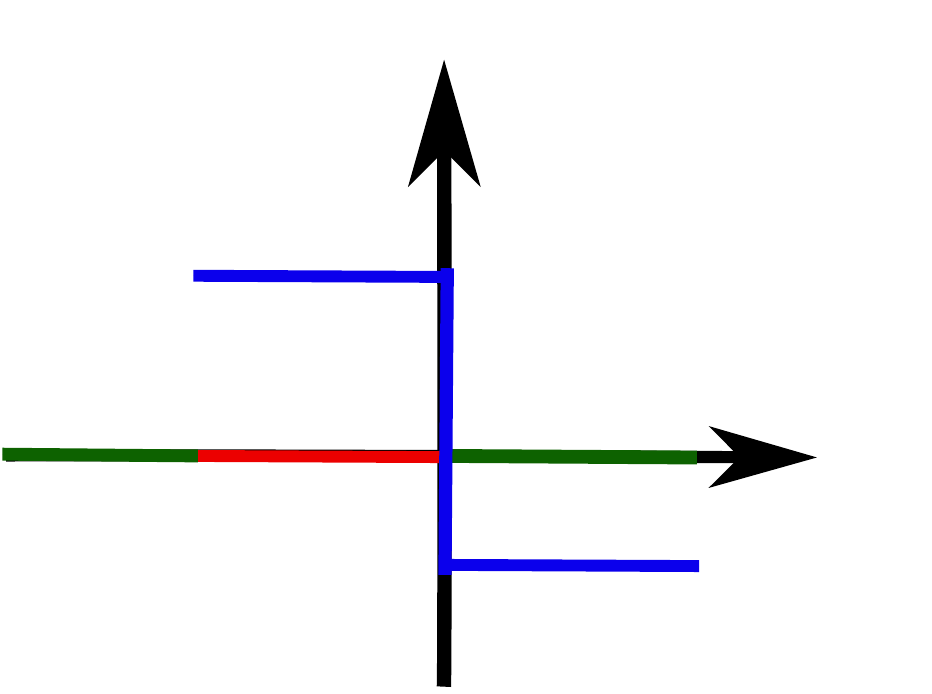}}%
    \put(0.18567641,0.18021467){\color[rgb]{1,0,0}\makebox(0,0)[lt]{\lineheight{1.25}\smash{\begin{tabular}[t]{l}$X_o$\end{tabular}}}}%
    \put(0.51517687,0.28088898){\color[rgb]{0.2,0.50196078,0}\makebox(0,0)[lt]{\lineheight{1.25}\smash{\begin{tabular}[t]{l}$X_u$\end{tabular}}}}%
    \put(0.79922026,0.161731){\color[rgb]{0,0,0}\makebox(0,0)[lt]{\lineheight{1.25}\smash{\begin{tabular}[t]{l}$x$\end{tabular}}}}%
    \put(0.34663472,0.65831664){\color[rgb]{0,0,1}\makebox(0,0)[lt]{\lineheight{1.25}\smash{\begin{tabular}[t]{l}$F(x)$\end{tabular}}}}%
    \put(0.35216283,0.08889905){\color[rgb]{0,0,1}\makebox(0,0)[lt]{\lineheight{1.25}\smash{\begin{tabular}[t]{l}-1\end{tabular}}}}%
    \put(0.49181254,0.41088487){\color[rgb]{0,0,1}\makebox(0,0)[lt]{\lineheight{1.25}\smash{\begin{tabular}[t]{l}2\end{tabular}}}}%
    \put(0,0){\includegraphics[width=\unitlength,page=2]{drawing.pdf}}%
  \end{picture}%
\endgroup%
 
 \caption{ The graph of $F$ and the initial and unsafe sets $(X_o,X_u)$.  \label{FigFigo1}}
 \end{center}
\end{figure}

Let the barrier function candidate $B(x) := x(x+2)$. As a result, it follows that  
$ K  = [-2,0]$ and 
$$ U(K)\backslash K = (-2 - \sigma,-2) \cup (0,\sigma),  $$ 
for some  $\sigma > 0$.

We start noting that $F$ satisfies Assumption \ref{ass1} and that 
\begin{align*}
\langle \nabla B(x), F(x) \rangle =  \begin{cases}
-2(x+1) & x \in (0,\sigma) \\
4(x+1) & x \in (-2-\sigma,-2).
\end{cases}
\end{align*}
Hence, 
\begin{align*}
    \langle \nabla B(x), F(x) \rangle \leq 0 \qquad \forall x \in U(\partial K) \backslash K.
\end{align*}
Thus,  \textcolor{blue}{using Lemma \ref{lemA8}},  we conclude that $\Sigma$ is safe with respect to $(X_o,X_u)$.  

Now, we consider the perturbed version $\Sigma_\varepsilon$ when
 $\varepsilon(x) = 1$ for all $x \in \mathbb{R}$ (see Fig. \ref{FigFigo3}).  
 
\begin{figure}
\begin{center}
  \def\svgwidth{1.3\columnwidth}
\begingroup%
  \makeatletter%
  \providecommand\color[2][]{%
    \errmessage{(Inkscape) Color is used for the text in Inkscape, but the package 'color.sty' is not loaded}%
    \renewcommand\color[2][]{}%
  }%
  \providecommand\transparent[1]{%
    \errmessage{(Inkscape) Transparency is used (non-zero) for the text in Inkscape, but the package 'transparent.sty' is not loaded}%
    \renewcommand\transparent[1]{}%
  }%
  \providecommand\rotatebox[2]{#2}%
  \newcommand*\fsize{\dimexpr\f@size pt\relax}%
  \newcommand*\lineheight[1]{\fontsize{\fsize}{#1\fsize}\selectfont}%
  \ifx\svgwidth\undefined%
    \setlength{\unitlength}{447.37106824bp}%
    \ifx\svgscale\undefined%
      \relax%
    \else%
      \setlength{\unitlength}{\unitlength * \real{\svgscale}}%
    \fi%
  \else%
    \setlength{\unitlength}{\svgwidth}%
  \fi%
  \global\let\svgwidth\undefined%
  \global\let\svgscale\undefined%
  \makeatother%
  \begin{picture}(1,0.46495883)%
    \lineheight{1}%
    \setlength\tabcolsep{0pt}%
    \put(0,0){\includegraphics[width=\unitlength,page=1]{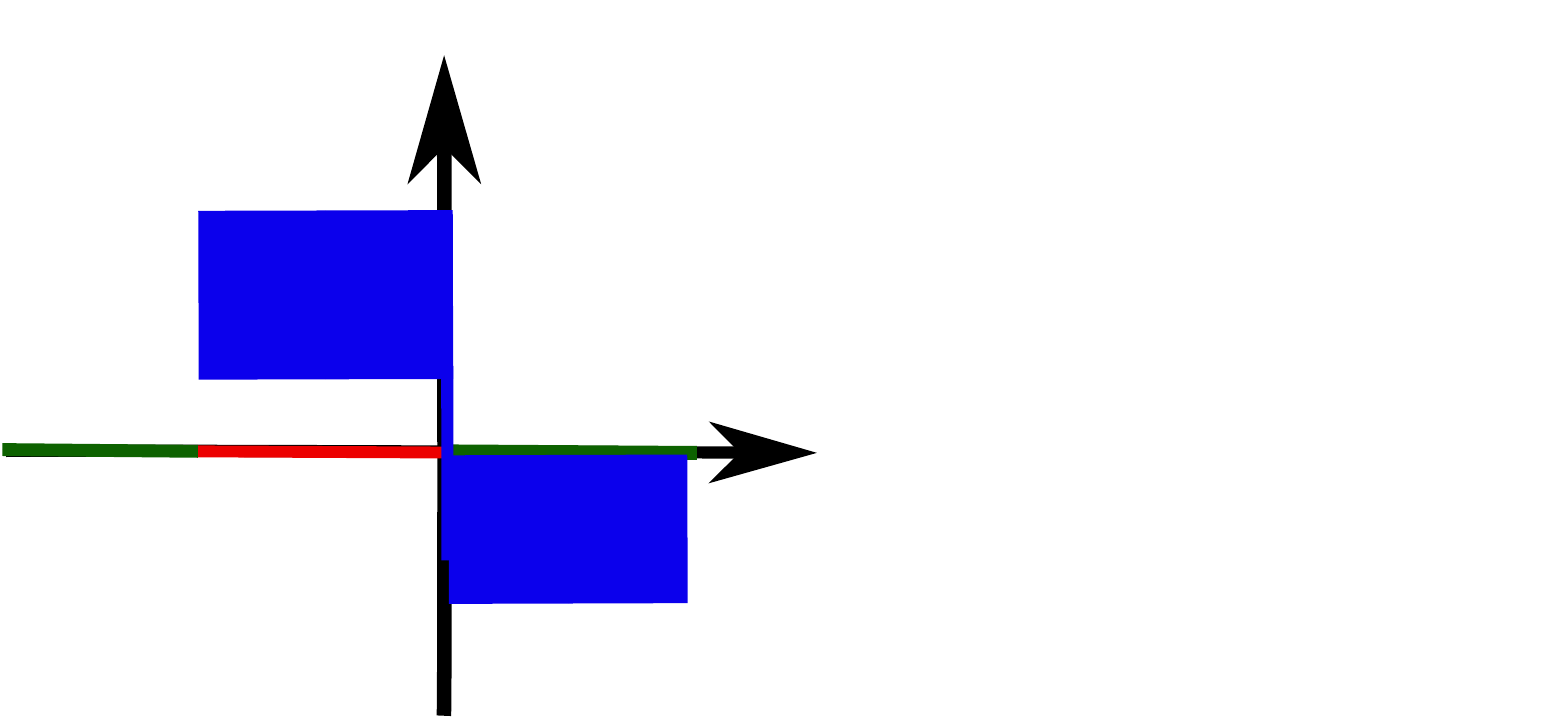}}%
    \put(0.12906141,0.11128352){\color[rgb]{1,0,0}\makebox(0,0)[lt]{\lineheight{1.25}\smash{\begin{tabular}[t]{l}$X_o$\end{tabular}}}}%
    \put(0.29416449,0.18936129){\color[rgb]{0.2,0.50196078,0}\makebox(0,0)[lt]{\lineheight{1.25}\smash{\begin{tabular}[t]{l}$X_u$\end{tabular}}}}%
    \put(0.45725834,0.11618207){\color[rgb]{0,0,0}\makebox(0,0)[lt]{\lineheight{1.25}\smash{\begin{tabular}[t]{l}$x$\end{tabular}}}}%
    \put(0,0){\includegraphics[width=\unitlength,page=2]{drawing1.pdf}}%
    \put(0.2956712,0.29663964){\color[rgb]{0,0,0}\makebox(0,0)[lt]{\lineheight{1.25}\smash{\begin{tabular}[t]{l}3\end{tabular}}}}%
    \put(0,0){\includegraphics[width=\unitlength,page=3]{drawing1.pdf}}%
    \put(0.20915691,0.03900307){\color[rgb]{0,0,0}\makebox(0,0)[lt]{\lineheight{1.25}\smash{\begin{tabular}[t]{l}-2\end{tabular}}}}%
    \put(0.30558079,0.41335544){\color[rgb]{0,0,1}\makebox(0,0)[lt]{\lineheight{1.25}\smash{\begin{tabular}[t]{l}$F(x)+\mathbb{B}$\end{tabular}}}}%
  \end{picture}%
\endgroup%
 
 \caption{ The graph of $x \mapsto F(x) + \mathbb{B}$ and the initial and unsafe sets $(X_o,X_u)$. \label{FigFigo3}}
 \end{center}
\end{figure} 
 
 We note that,  for all $\mu \in \mathbb{B}$,
\begin{align*}
    \langle \nabla B(x), & ~ F(x) + \mu \rangle  \\ & =
   \begin{cases}
     2(x+1)(-1 + \mu) & \forall x \in (0,\sigma) 
     \\
     2(x+1)(2+ \mu) & \forall x \in (-2-\sigma,-2).
     \end{cases}
\end{align*}
The latter implies that
\begin{align*}
 \langle \nabla B(x), F(x) + \mu \rangle  \leq 0 \qquad \forall x \in U(\partial K)\backslash K, ~~ \forall \mu \in \mathbb{B}.
\end{align*}
Therefore,  \textcolor{blue}{using Lemma \ref{lemA8}},  we conclude that $\Sigma_{\varepsilon}$ is safe with respect to $(X_o,X_u)$.  Hence, $\Sigma$ is robustly safe with respect to $(X_o,X_u)$.

At this point, we show that $\Sigma$ is not strongly robustly safe with respect to $(X_o,X_u)$. Namely, we show that for any constant perturbation $\varepsilon$, there is a solution to $\Sigma^s_\varepsilon$ that starts from and leaves $X_o$. 

\begin{figure} 
\begin{center}
  \def\svgwidth{3\columnwidth}
\begingroup%
  \makeatletter%
  \providecommand\color[2][]{%
    \errmessage{(Inkscape) Color is used for the text in Inkscape, but the package 'color.sty' is not loaded}%
    \renewcommand\color[2][]{}%
  }%
  \providecommand\transparent[1]{%
    \errmessage{(Inkscape) Transparency is used (non-zero) for the text in Inkscape, but the package 'transparent.sty' is not loaded}%
    \renewcommand\transparent[1]{}%
  }%
  \providecommand\rotatebox[2]{#2}%
  \newcommand*\fsize{\dimexpr\f@size pt\relax}%
  \newcommand*\lineheight[1]{\fontsize{\fsize}{#1\fsize}\selectfont}%
  \ifx\svgwidth\undefined%
    \setlength{\unitlength}{1105.81372467bp}%
    \ifx\svgscale\undefined%
      \relax%
    \else%
      \setlength{\unitlength}{\unitlength * \real{\svgscale}}%
    \fi%
  \else%
    \setlength{\unitlength}{\svgwidth}%
  \fi%
  \global\let\svgwidth\undefined%
  \global\let\svgscale\undefined%
  \makeatother%
  \begin{picture}(1,0.17850365)%
    \lineheight{1}%
    \setlength\tabcolsep{0pt}%
    \put(0,0){\includegraphics[width=\unitlength,page=1]{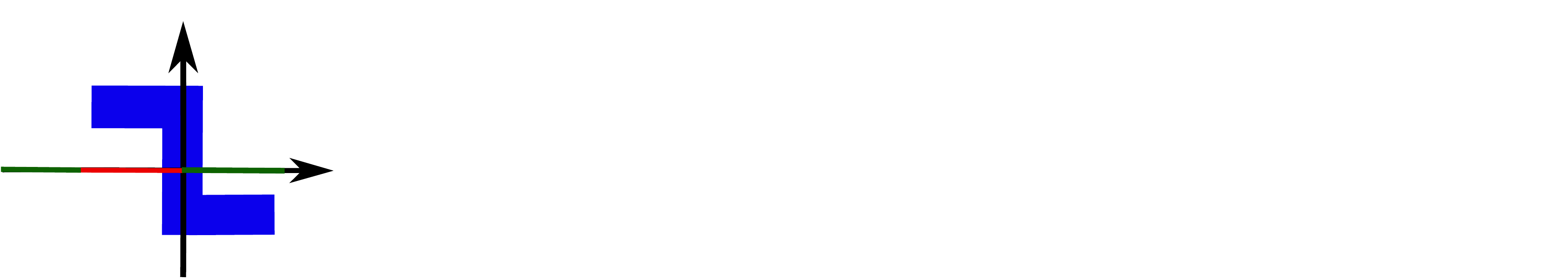}}%
    \put(0.03740684,0.0511122){\color[rgb]{1,0,0}\makebox(0,0)[lt]{\lineheight{1.25}\smash{\begin{tabular}[t]{l}$X_o$\end{tabular}}}}%
    \put(0.12711543,0.07862159){\color[rgb]{0,0.50196078,0}\makebox(0,0)[lt]{\lineheight{1.25}\smash{\begin{tabular}[t]{l}$X_u$\end{tabular}}}}%
    \put(0.19511897,0.04699416){\color[rgb]{0,0,0}\makebox(0,0)[lt]{\lineheight{1.25}\smash{\begin{tabular}[t]{l}$x$\end{tabular}}}}%
    \put(0,0){\includegraphics[width=\unitlength,page=2]{drawing2.pdf}}%
    \put(0.0857461,0.02551831){\color[rgb]{0,0,0}\makebox(0,0)[lt]{\lineheight{1.25}\smash{\begin{tabular}[t]{l}-1\end{tabular}}}}%
    \put(0.12865659,0.10390941){\color[rgb]{0,0,0}\makebox(0,0)[lt]{\lineheight{1.25}\smash{\begin{tabular}[t]{l}2\end{tabular}}}}%
    \put(0.12227316,0.15762684){\color[rgb]{0,0,1}\makebox(0,0)[lt]{\lineheight{1.25}\smash{\begin{tabular}[t]{l}$F(x + \varepsilon(x)\mathbb{B})+\varepsilon(x)\mathbb{B}$\end{tabular}}}}%
  \end{picture}%
\endgroup%
 
 \caption{ The graph of $x \mapsto F(x + \varepsilon(x) \mathbb{B}) + \varepsilon(x) \mathbb{B}$ and the initial and unsafe sets $(X_o,X_u)$.   \label{FigFigo2} }
 \end{center}
\end{figure}

Note that we can reason on constant perturbations since $X_o$ is compact.  

For each $\varepsilon>0$, we introduce the system
\begin{equation*}
\begin{aligned}
\dot{x} \in \begin{cases} 
2 + \varepsilon \mathbb{B}  & \text{if} \;\ x + \varepsilon \mathbb{B}  \leq 0 \\
[-1,2] + \varepsilon \mathbb{B} & \text{if} \;\ x + \varepsilon \mathbb{B} = 0 \\
-1 + \varepsilon \mathbb{B} & \text{if} \;\ x + \varepsilon \mathbb{B} \geq 0. 
\end{cases}  
\end{aligned}
\end{equation*}
The previous system can be further expressed as (see Fig. \ref{FigFigo2})
\begin{equation}
\label{eqce}
\begin{aligned}
\dot{x} \in  
 \begin{cases}
[2 - \varepsilon, 2 + \varepsilon]  & \text{if} \;\ x \leq - \varepsilon \\
[-1 - \varepsilon, 2 + \varepsilon] & \text{if} \;\ x \in [-\varepsilon, \varepsilon] \\
[-1 - \varepsilon, -1 + \varepsilon] & \text{if} \;\ x \geq \varepsilon.
\end{cases}
\end{aligned}
\end{equation}
\textcolor{blue}{Now,  for any $\varepsilon >0$,  we consider the function $\phi_\varepsilon : [0,\varepsilon] \rightarrow \mathbb{R}^n$ satisfying $\phi_\varepsilon(0) = 0$ and  $\dot{\phi}_\varepsilon(t) = 1 \in 
F( \phi_\varepsilon(t) + \varepsilon( \phi_\varepsilon(t) ) \mathbb{B} ) + \varepsilon( \phi_\varepsilon(t) ) \mathbb{B}$.   Hence,  $\phi_\varepsilon$ is solution to \eqref{eqce},  starting from $\partial X_o$,  that leaves $X_o$.  Hence, for any $\varepsilon > 0$, $\Sigma^{s}_{\varepsilon}$ is not safe.  }
\end{example}

\section{Strong Robust Safety}  \label{sec.4}

\textcolor{blue}{In the next section, we propose sufficient conditions, in terms of barrier functions, to guarantee strong robust  safety.  Furthermore,  we address the converse problem.}

We start recalling a tool that allows us to use  locally-Lipschitz barrier functions \cite{clarke2008nonsmooth}.

\begin{definition} [Clarke generalized gradient] \label{defgen}
\blue{Let $B : \mathbb{R}^n \rightarrow \mathbb{R}$ be locally Lipschitz.  Let $\Omega \subset \mathbb{R}^n$ be any null-measure set,  and let $\Omega_B \subset \mathbb{R}^n$ be the set of points at which $B$ fails to be differentiable. The Clarke generalized gradient of $B$ at $x$ is the set-valued map 
$\partial_C B : \mathbb{R}^n \rightrightarrows \mathbb{R}^n$  given by
\begin{align*}
\partial_C B(x) := \co \left\{ \lim_{i \rightarrow \infty} \nabla B(x_i) : x_i \rightarrow x,~x_i \in \Omega \backslash \Omega_B \right\}.
\end{align*}
where $\co(\cdot)$ is the convex hull of the elements in $(\cdot)$}.  
\end{definition} 

\begin{remark}
When $B$ is locally Lipschitz, then $\partial_C B$ is upper semicontinuous and its images are nonempty, compact, and convex. 
\end{remark}

\subsection{Sufficient Conditions}

In the following result,  we show that the sufficient conditions for robust safety in \cite{RubSafPI} are strong enough to guarantee strong robust safety under Assumption \ref{ass1}. 

\begin{theorem} \label{thm1}
Consider system $\Sigma$ such that Assumption \ref{ass1} holds.
Let $B : \mathbb{R}^n \rightarrow \mathbb{R}$ be a barrier function candidate with respect to $(X_o, X_u) \subset \mathbb{R}^n \times \mathbb{R}^n$. Then, $\Sigma$ is strongly robustly safe with respect to $(X_o,X_u)$ if one the following conditions holds:
\begin{enumerate} [label={C\ref{thm1}\arabic*.},leftmargin=*]
\item \label{item:cond1} There exists a continuous function $\varepsilon : \mathbb{R}^n \rightarrow \mathbb{R}_{>0}$ such that the set $K$ in \eqref{eq.4} is forward invariant for $\Sigma^s_{\varepsilon}$.
\item \label{item:cond1bis}   $\Sigma$ is robustly safe  with respect to $(X_o, X_u)$ and $F$ is continuous. 

\item \label{item:cond3} The function $B$ is locally Lipschitz and 
\begin{align} \label{eqlipadd}
\hspace{-0.4cm} \langle  \partial_C B(x), \eta \rangle  \subset  \mathbb{R}_{<0} \quad \forall \eta \in F(x),  ~  \forall x \in  \partial K.
\end{align}
\item \label{item:cond2}  $B$ is continuously differentiable and \eqref{eqrsafe} holds. 

\end{enumerate}
\end{theorem}

\begin{proof}
\blue{We start noting that establishing strong robust safety under \ref{item:cond2} would follow  straightforwardly   if we prove robust safety under \ref{item:cond3}.  Indeed,  when $B$ is continuously differentiable,   $\partial_C B = \nabla B$. 

We now prove strong robust safety  under each of the statements in \ref{item:cond1}-\ref{item:cond3}. }
\begin{itemize}
\item Under \ref{item:cond1},   we conclude the existence of a continuous function $\varepsilon : \mathbb{R}^n \rightarrow \mathbb{R}_{>0}$ such that the solutions to $\Sigma^s_{\varepsilon}$ starting from $X_o$ never reach $X_u$, which by definition implies strong robust safety of $\Sigma$ with respect to $(X_o,X_u)$. 

\item  Under \ref{item:cond1bis},     there exists $\varepsilon : \mathbb{R}^n \rightarrow \mathbb{R}_{>0}$ continuous such that $\Sigma_{\varepsilon}$ is safe.  Using Lemma \ref{lemA333} in Appendix II,  we conclude the existence of $\delta : \mathbb{R}^n \rightarrow \mathbb{R}_{>0}$ continuous such that 
$$    F(x + \delta(x) \mathbb{B})  \subset F(x) + (\varepsilon(x)/2) \mathbb{B} \qquad \forall x \in \mathbb{R}^n,      $$
which implies that 
$$    F(x + \delta(x)  \mathbb{B})  + (\varepsilon(x)/2) \mathbb{B}   \subset F(x) + \varepsilon(x) \mathbb{B} \qquad \forall x \in \mathbb{R}^n.      $$
Now,  since $F$ is convex, we conclude that 
$$  \co \{  F(x + \delta(x)  \mathbb{B}) \} + (\varepsilon(x)/2) \mathbb{B}   \subset F(x) + \varepsilon(x) \mathbb{B} \qquad \forall x \in \mathbb{R}^n.      $$
Hence,  $\Sigma^s_{\varepsilon_o}$ defined as in \eqref{eq3}, with $\varepsilon_o := \min \{\delta, \varepsilon/2 \}$,   is safe.
 
\item Under \ref{item:cond3},   we propose to grid the boundary $\partial K$ using a sequence of nonempty compact subsets $\{ D_i\}_{i=1}^{N}$, where $N \in \{1,2,3, \dots ,\infty \}$. 
We assume that 
\begin{align*}
\bigcup^{N}_{i = 1} D_i  = \partial K, \qquad D_i  \subset \partial K \quad  \forall i \in \{1,2, \dots ,N\}.
\end{align*}
Furthermore,  we choose the sequence 
$\{ D_i\}_{i=1}^{N}$ such that,  for each $i \in \{1,2,...,N\}$,  we can find a finite set  $\mathcal{N}_i \subset \{1,2,...,N \}$  
such that 
\begin{align*}
D_i \cap D_j & = \emptyset \qquad \forall  j \notin \mathcal{N}_i,
\\
\mbox{int}_{\partial K}(D_i \cap D_j) & = \emptyset \qquad \forall j \in \mathcal{N}_i,
\end{align*}
where $\mbox{int}_{\partial K}(\cdot)$ denotes the  interior of $(\cdot)$ relative to $\partial K$.  
\blue{Such a decomposition always exists according to  Whitney covering lemma \cite{10.2307/1989708}.}

Next,  we introduce the following claim.
\begin{claim} \label{clm1}
 For each $i \in \{1,2,...,N\}$, there exists $\delta_i>0$ such that,  for each $x \in  D_i$,
\begin{align}\label{eq00}
\langle  \partial_C B(x), \eta \rangle \subset (-\infty, 0) \qquad \forall \eta \in H(x, \delta_i),
\end{align}
where $H(x, \delta_i) := \co \{ F(x + \delta_i \mathbb{B})  \} + \delta_i \mathbb{B}$.
\end{claim}

\blue{To complete the proof,  under Claim \ref{clm1},  we let
\begin{equation}\label{eq0}
\begin{aligned}
\varepsilon_o(x) := \min \{ \delta_{i} : x \in D_i \}.
\end{aligned}
\end{equation}
Clearly,  since the sequence $\{ \delta_i\}^N_{i=1}$ is positive,  and by definition of the sequence $\{ D_i\}^N_{i=1}$,  we conclude that $\varepsilon_o$ is strictly positive and lower semicontinuous.  Finally,  \textcolor{blue}{using  \cite[Theorem 1]{katvetov1951real}},   we conclude the existence of a continuous function $\varepsilon : \partial K \rightarrow \mathbb{R}_{ > 0}$ such that $ \varepsilon(x) \leq \varepsilon_o(x)$ for all $x \in \partial K$.  As a result, we conclude that 
\begin{align}\label{eq00*}
\hspace{-0.7cm} \langle  \partial_C B(x), \eta \rangle \subset (-\infty, 0) ~~ \forall \eta \in H(x, \varepsilon(x)) ~~ \forall x \in \partial K.
\end{align}
As a result,  using \cite[Theorem 5]{draftautomatica},  we conclude that the system 
$\dot{x} \in H(x, \varepsilon(x))$,  with $x \in \mathbb{R}^n$,  is robustly safe; hence,  $\Sigma$ is strongly robustly safe}.  
\end{itemize}

Finally,  we prove Claim \ref{clm1}  using  contradiction.  
That is,  we assume the existence of a positive sequence $\{\sigma_{j} \}^\infty_{j=1}$ that converges to zero such that,  for each $j \in \{1,2,...\}$,  we can find $x_j \in D_i$, $\eta_j \in \partial_C B(x_j)$,  and $\zeta_j \in H(x_j, \sigma_j)$ such that 
$ \langle \eta_j, \zeta_j\rangle \geq 0$.
 
Since $D_i$ is compact, by passing to an appropriate sub-sequence,  we conclude  the existence of $x^* \in D_i$ such that $\lim_{j \rightarrow \infty} x_j = x^*$.
 
Next, since the set-valued map $\partial_C B$ is upper semicontinuous with nonempty and compact images, we conclude that by passing to an appropriate sub-sequence 
$$ \lim_{j \rightarrow \infty} \eta_j = \eta^* \in \partial_C B(x^*). $$ 

\textcolor{blue}{Furthermore, we note that $H$ is  the convex hull of the composition of $F$ and $(x,\delta) \mapsto x + \delta \mathbb{B}$  added to the map  $\delta \mapsto \delta \mathbb{B}$. } The latter map is continuous and $F$ verifies Assumption \ref{ass1}.  Hence, $H$ enjoys the same properties as $F$.  As a result,  we have  $ \lim_{j \rightarrow \infty} \zeta_j = \zeta^* \in H(x^*,0) = F(x^*)$.

Now,  since $\langle \eta_j, \zeta_j\rangle \geq 0$ for all $j \in \{1,2,... \}$,  then $\lim_{j \rightarrow \infty} 
 \langle \eta_j,  \zeta_j \rangle \geq 0$.  
 However,  this contradicts \eqref{eqlipadd}, which proves the claim. 
\end{proof}

\begin{remark}
\blue{  Note that we can show strong robust safety by showing robust safety for the system  $\dot{x} \in 
\co \{F(x + \varepsilon(x) \mathbb{B})\}$,  or by showing safety for   $\Sigma^s_\varepsilon$,  for some $\varepsilon$ continuous and positive.  However,  the condition we need to use in this case must involve the term 
$\varepsilon$,  hence,  one need to search for such a valid candidate $\varepsilon$ for which the condition holds.  As opposed to the proposed conditions,  which involve the nominal system $\Sigma$ only,  and induce the existence of a robustness  margin $\varepsilon$.  }
\end{remark}

\subsection{The Converse Problem}

In this section,  we establish the equivalence between strong robust safety and the existence of a continuously differentiable barrier functions candidate $B$ satisfying \eqref{eqrsafe}.   Thanks to the following intermediate result,  which establishes that we can always squeeze a continuous set-valued map between an upper semicontinuous map and its strongly perturbed version.   
\begin{proposition} \label{lemA444}
Consider a set-valued map $F : \mathbb{R}^n \rightrightarrows \mathbb{R}^n$ satisfying Assumption \ref{ass1}.  Then,  for each continuous function $\varepsilon : \mathbb{R}^n \to \mathbb{R}_{>0}$,  there exists a continuous set-valued map $G : \mathbb{R}^n \rightrightarrows \mathbb{R}^n$ with convex and compact images such that
\textcolor{blue}{\begin{align*}
    F(x) \subset G(x) \subset \co \{ F(x+\varepsilon(x) \mathbb{B}) \} + \varepsilon(x) \mathbb{B}  \qquad \forall x \in \mathbb{R}^n.
\end{align*}}
\end{proposition}
The latter proposition allows us to prove the converse strong robust-safety theorem,  using the converse robust-safety theorem established in \cite{RubSafPII},  under the following mild assumption: 
\begin{assumption} \label{ass3}
\textcolor{blue}{$\cl(X_o) \cap \cl (X_u) = \emptyset$.}
\end{assumption}
Different from the converse robust safety theorem in \cite{RubSafPII},  in this case,  we do not require $F$ to be continuous.  
\begin{theorem} \label{thm4}
Consider the differential inclusion $\Sigma$ in \eqref{eq1}, with  $F$ satisfying 
Assumption \ref{ass1}.  Consider the  initial and unsafe sets $(X_o, X_u) \subset   \mathbb{R}^n \times \mathbb{R}^n$ such that Assumption \ref{ass3} holds. 
Then,  the following statements are equivalent:
\begin{enumerate} [label={S\ref{thm4}\arabic*.},leftmargin=*]
\item \label{item:thm4_S1} $\Sigma$ is  strongly robustly safe with respect to $(X_o,X_u)$.
\item  \label{item:thm4_S2} There exists a continuously differentiable barrier function candidate $B$ such that  \eqref{eqrsafe} holds. 
\end{enumerate} 
\end{theorem}

\begin{proof}  The proof that \ref{item:thm4_S2} implies \ref{item:thm4_S1} follows from Theorem \ref{thm1}. 

To prove that \ref{item:thm4_S1} implies \ref{item:thm4_S2},  we assume that $\Sigma$ is strongly robustly safe  with respect to $(X_o,X_u)$ and we let $\varepsilon$ be a strong robust-safety margin.  By virtue of Proposition  \ref{lemA444},  we conclude the existence of  a continuous set-valued map $G : \mathbb{R}^n \rightrightarrows \mathbb{R}^n$ with convex and compact images such that
\begin{align*}
F(x) \subset G(x) \subset \co \{ F(x + \varepsilon(x) \mathbb{B}) \} +  (\varepsilon(x)/2) \mathbb{B} ~~~ \forall  x \in \mathbb{R}^n.
\end{align*}
Note that strong robust safety of $\Sigma$ implies robust safety of 
$$ \Sigma_g :   \dot{x} \in G(x) \quad x \in \mathbb{R}^n $$
 with $\varepsilon/2$ as a robustness margin; namely,  the system 
\begin{equation*}
\begin{aligned}
\Sigma_{g \varepsilon}: \dot{x} \in G(x) + (\varepsilon(x)/2) \mathbb{B}  \qquad  x \in \mathbb{R}^n
\end{aligned}
\end{equation*}
is safe.
Hence,  using  \cite{RubSafPII}, we conclude that the existence of a continuously differentiable barrier function candidate $B$ such that 
$ \langle \nabla B (x), \eta \rangle < 0$ for all $(\eta,x) \in (G(x),  \partial K).  $
The latter completes the proof since $F \subset G$. 
\end{proof}

\section{Uniform Strong Robust Safety} \label{sec.5}

In this section,  we strengthen the sufficient conditions in Theorem \ref{thm1} to conclude uniform strong robust safety.   In the following theorem, we analyze particular situations,  where uniform strong robust safety is equivalent to strong robust safety.  

\begin{theorem} \label{thm2}
Consider the differential inclusion $\Sigma$ such that Assumption \ref{ass1} holds. 
Let $B : \mathbb{R}^n \rightarrow \mathbb{R}$ be a barrier function candidate with respect to $(X_o, X_u)$. Then, $\Sigma$ is uniformly strongly robustly safe with respect to $(X_o,X_u)$ if one of the following conditions holds:
\begin{enumerate} [label={C\ref{thm2}\arabic*.},leftmargin=*]
\item \label{item:cond2uu} 
There exists a continuous function $\varepsilon : \mathbb{R}^n \rightarrow \mathbb{R}_{>0}$ such that the set $K$ is forward invariant for 
$\Sigma^s_\varepsilon$ and the set $\partial K$ is bounded.
\item \label{item:cond3uu} 
$\Sigma$ is strongly robustly safe with respect to $(X_o,X_u)$ and either $\mathbb{R}^n \backslash X_u$ or $\mathbb{R}^n \backslash X_o$ is bounded.
\end{enumerate}
\end{theorem}

\begin{proof}
\begin{itemize}
\item Under \ref{item:cond2uu} and when $\partial K$ is bounded,   \blue{ we show that 
$\epsilon^* := \inf \{ \epsilon(x) : x \in U(\partial K) \} > 0$,  for $U(\partial K)$ a neighborhood of $\partial K$,   is a strong robust-safety margin using contradiction.  Indeed,  assume that there exists a solution 
$\phi$ to $\Sigma^s_{\epsilon^*}$ starting from $x_o \in X_o$ that reaches the set $X_u$ in finite time.  
Namely,  there exists $t_1$,  $t_2 \in \dom \phi$ such that $t_2 > t_1 \geq 0$,  $\phi(t_2) \in U(\partial K) \backslash K$,  $\phi(t_1) \in U(\partial K) \cap K$, 
$\phi(0) \in X_o$,  and $\phi([t_1,t_2]) \subset U(\partial K)$.  The latter implies that $\phi$,  restricted to the interval $[t_1,t_2]$,  is also a solution to 
$\Sigma^s_\epsilon$.   However, since the set $K$ is forward invariant for $\Sigma^s_{\epsilon}$,  we conclude that the solution $\phi([t_1,t_2])$ must lie within the set $K$, which yields to a contradiction.   }

\item  Under \ref{item:cond3uu}, we conclude the existence of a continuous function $\varepsilon : \mathbb{R}^n \rightarrow \mathbb{R}_{>0}$ such that $\Sigma^s_\varepsilon$ is safe with respect to $(X_o,X_u)$. Hence, the set 
$$ K_{\varepsilon} := \{ \phi(t) : t \in \dom \phi,~\phi \in \mathcal{S}_{\Sigma^s_\varepsilon}(x),~ x \in X_o \},  $$ 
where $\mathcal{S}_{\Sigma^s_\varepsilon}(x)$ is the set of maximal solutions to $\Sigma^s_\varepsilon$ starting from $x$, is forward invariant for $\Sigma^s_\varepsilon$, $K_{\varepsilon} \cap X_u = \emptyset$, and $X_o \subset K_{\varepsilon}$. 
Furthermore, when either the complement of $X_u$ is bounded or the complement of $X_o$ is bounded, we conclude that $\partial K_{\varepsilon}$ is bounded. Hence, the rest of the proof follows as in the proof of \ref{item:cond2uu}. 
\end{itemize}
\end{proof}

In the next theorem,  we propose a set of infinitesimal conditions to guarantee uniform strong robust safety without using the boundedness assumptions in \ref{item:cond2uu} and \ref{item:cond3uu}.  
Before doing so,  we recall a tool that allows us to consider semicontinuous barrier functions \cite{clarke2008nonsmooth}. 
\begin{definition}[Proximal subdifferential]  \label{defps}
The proximal subdifferential of $B : \mathbb{R}^n \rightarrow \mathbb{R}$ at $x$ is the set-valued map 
$\partial_P B : \mathbb{R}^n \rightrightarrows \mathbb{R}^n$ given by
\begin{align*}
\partial_P B(x) := \left\{ \zeta \in \mathbb{R}^n : [\zeta^\top~-1]^\top \in N^P_{\epi B} (x, B(x)) \right\},
\end{align*}
where $ \epi B := \left\{(x, r) \in \mathbb{R}^n \times \mathbb{R} : r \geq B(x) \right\}$ \blue{and,  given  a closed subset $S \subset \mathbb{R}^{n+1}$ and $y \in \mathbb{R}^{n+1}$, 
$$ N_S^P(y) := \left\{ \zeta \in \mathbb{R}^{n+1} : \exists r > 0~\mbox{s.t.}~|y+r\zeta|_{S} = r 
|\zeta| \right\}. $$}
\end{definition}
\begin{remark}
\blue{When $B$ is twice continuously differentiable at 
$x \in \mathbb{R}^n$,    then 
$\partial_P B(x) = \left\{ \nabla B(x) \right\}$.
Moreover,  the latter equality is also true  if $B$ is continuously differentiable at $x$ and $\partial_P B(x) \neq \emptyset$. }
\end{remark}
Furthermore,  we consider the following assumption:
\begin{assumption}\label{assusrs}
\blue{There exists  $\lambda_2 : \mathbb{R}^n \to \mathbb{R}_{\geq 0}$ continuous,  which is known,  and there exists  $\lambda_1 : \mathbb{R} \to \mathbb{R}_{\geq 0} $ continuous with 
$\lambda_1(0) = 0$ such that
\begin{align} \label{eqass}
\hspace{-0.3cm} F(x+\delta \mathbb{B}) \subset F(x) + \lambda_1(\delta) \lambda_2(x) \mathbb{B} \quad \forall x \in \partial K,  ~ \forall  \delta > 0.
\end{align} }
\end{assumption}

Note that Assumption \ref{assusrs} implies continuity of $F$ in the set-valued sense.   Furthermore,  we will show in the following result that the continuity of $F$ implies the existence of continuous functions $\lambda_1 : \mathbb{R} \to \mathbb{R}_{\geq 0} $ and 
$\lambda_2 : \mathbb{R}^n \to \mathbb{R}_{\geq 0}$, with $\lambda_1(0) = 0$,  such that  \eqref{eqass} holds.  However,  we do not guarantee the explicit knowledge  of  $\lambda_2$ unless the following extra assumption is made.

\begin{assumption} \label{assusrs+}
There exists  
$g_o : (-\infty , \bar{a}] \to [0,+\infty)$ nonincreasing and continuous,  \blue{which is known},  such that 
\begin{align*} 
\graph (g_o)  := \{(a,b) \in \mathbb{R} \times \mathbb{R}  :  c(a,b) = -b \}, 
\end{align*}
where $\bar{a} := \sup \{a \in \mathbb{R} : c(a,0)=0\}$, 
$$c(a,b) := \log \left( \max_{ \{|y|< \exp(a),  s \leq \exp(b)\} } g(y,s) \right),   $$
and  $g(y,s) := |F(y+s \mathbb{B}) - F(y)|_H - |F(s \mathbb{B}) - F(0)|_H. $
\end{assumption}

The following result summarizes the aforementioned discussion.  The proof is in the Appendix. 

\begin{proposition} \label{lem2p} 
For every continuous set-valued map $F:\mathbb{R}^n \rightrightarrows \mathbb{R}^n$, there exist two continuous functions $\lambda_1 : \mathbb{R} \to \mathbb{R}_{\geq 0} $ and $\lambda_2 : \mathbb{R}^n \to \mathbb{R}_{\geq 0}$, with $\lambda_1(0) = 0$, such that  \eqref{eqass} holds.     Moreover,  when Assumption \ref{assusrs+} holds.  Then, we can take 
$ \lambda_2 (x) := \exp \left( \max\{g,h\} (\ln{|x|}) \right) + 1,  $
where 
\begin{equation*}
    \begin{aligned}
        h(\ln{|x|}) & := \sup \{ c(a,\ln{|x|}) - g(a)  : a \in \mathbb{R} \},
\\
 g(a) & :=   
 \begin{cases}
-\frac{1}{2}g_o(a)  & \text{if} \;\ a \leq \bar{a},
\\
c(a,a) - c(\bar{a},\bar{a}) + a -\bar{a} & \text{if} \;\ a > \bar{a}. 
\end{cases}
\end{aligned}
\end{equation*}
\end{proposition} 

At this point,  we present the main result of this section. 

\begin{theorem} \label{thm3}
Consider the differential inclusion $\Sigma$ such that Assumptions \ref{ass1} and \ref{assusrs} hold. 
Let $B : \mathbb{R}^n \rightarrow \mathbb{R}$ be a barrier function candidate with respect to $(X_o, X_u)$. 
Then, $\Sigma$ is uniformly strongly robustly safe with respect to $(X_o,X_u)$ if one of the following conditions holds:

\begin{enumerate} [label={C\ref{thm3}\arabic*.},leftmargin=*]

\item \label{item:cond1uuu}  $B$ is continuously differentiable and 
\begin{align*}
\hspace{-0.8cm} \inf \left\{ \frac{\langle \nabla B (x), -F(x) \rangle / \lvert \nabla B(x) \rvert}{1+ \lambda_2(x)} :  x \in  \partial K  \right\} > 0.
\end{align*}

\item \label{item:cond2uuu} 
$B$ is locally Lipschitz and
\begin{align*}
\hspace{-0.8cm} \inf \left\{ \frac{\langle \zeta , -F(x) \rangle/\lvert \zeta \rvert}{1 +  \lambda_2(x)} : (\zeta,x) \in  \partial_C B(x)  \times \partial K \right\} > 0.
 \end{align*}

\item \label{item:cond3uuu} 
 $B$ is lower semicontinuous, $F$ is locally Lipschitz, and  
\begin{align*}
\hspace{-0.8cm} \inf \left\{ \frac{\langle \zeta , -F(x) \rangle/\lvert \zeta \rvert}{1 + \lambda_2(x)}:
 (\zeta,x) \in   \partial_P B (x) \times  U(\partial K) \right\}  > 0.
\end{align*}

\item \label{item:cond4uuu} $B$ is upper semicontinuous, $F$ is locally Lipschitz, $\cl(K) \cap X_u = \emptyset$, and  
\begin{align*}
\hspace{-0.8cm} \inf \left\{ \frac{\langle \zeta, F(x) \rangle/\lvert \zeta\rvert}{1 + \lambda_2(x)}  :  (\zeta,x) \in   \partial_P (-B(x)) \times  U(\partial K)  \right\}  > 0.
\end{align*}
\end{enumerate}
\end{theorem} 

\begin{proof}
\blue{The proof under \ref{item:cond1uuu} would follow  straightforwardly  if we prove strong robust safety under \ref{item:cond2uuu}.  Indeed,  when $B$ is continuously differentiable,   $\partial_C B = \nabla B$. }

\begin{itemize}

\item Under \ref{item:cond2uuu},  we let $\varepsilon>0$ and
\begin{align*}
\hspace{-0.2cm} \sigma := \inf \left\{ \frac{\langle \zeta, - F(x) \rangle}{\lvert \nabla B(x) \rvert (1 +  \lambda_2(x)) } : x \in  \partial K, ~ \zeta \in \partial_C B (x) \right\}.
\end{align*} 
Then,  for all $(x,\mu )\in (\partial K ,\mathbb{B})$, we use \eqref{eqass} to conclude that, for each $\eta_1 \in F(x+\varepsilon \mu)$, we can find $\eta_2 \in F(x)$ such that $ |\eta_1 - \eta_2| \leq   \lambda_1(\varepsilon) \lambda_2(x)$.    As a result,   for each $\zeta \in \partial_C B (x)$,  we have 
\begin{align*}
& \langle \zeta,   \eta_1 + \varepsilon \mu \rangle  \\ & =  
\langle \zeta,  \eta_2 \rangle + \langle \zeta,  , \varepsilon \mu \rangle   +
\langle \zeta,  \eta_1 - \eta_2 \rangle 
\\ & \leq  - \sigma |\zeta| (1 + \lambda_2(x) )  +  
 |\zeta| ( \varepsilon +   \lambda_1(\varepsilon)\lambda_2(x) )
\\ &  \leq  (-\sigma + \max(\varepsilon, \lambda_1(\varepsilon)))  |\zeta| 
(1 + \lambda_2(x)).
\end{align*} 
Hence,  by choosing $\varepsilon >0$ sufficiently small,  we guarantee that 
$ - \sigma +\max (\varepsilon, \lambda_1(\varepsilon)) < 0.  $
Therefore,  for each $\zeta \in \partial_C B (x)$,  we have 
\begin{align*}
    \langle \zeta,  F(x+\varepsilon \mathbb{B}) + \varepsilon \mathbb{B} \rangle \subset   \mathbb{R}_{<0} \qquad \forall x \in \partial K.
\end{align*}
\blue{and,  thus,  for each $\zeta \in \partial_C B (x)$,  we have 
\begin{align*}
    \langle \zeta,  \co \{ F(x+\varepsilon \mathbb{B}) \} + \varepsilon \mathbb{B} \rangle \subset   \mathbb{R}_{<0} \qquad \forall x \in \partial K.
\end{align*}}
This implies forward invariance of the set $K$ for 
$\Sigma^s_\varepsilon$ using \cite[Theorem 6]{draftautomatica},  which is enough to conclude uniform strong robust safety. 

\item  Under \ref{item:cond3uuu},  when $B$ is lower semicontinuous,  we show the existence of a constant 
$\varepsilon>0$ such that 
\begin{align}  \label{eqdecrease1}
\hspace{-0.8cm}  \langle \partial_P B(x),  \co \{ F(x+\varepsilon \mathbb{B})  \} + \varepsilon \mathbb{B} \rangle \subset \mathbb{R}_{<0}   ~  \forall x \in  U(\partial K).
\end{align}
Now,  \textcolor{blue}{using Lemma \ref{lemA8}},  we conclude that \eqref{eqdecrease1} implies that,  along each solution $\phi$ to $\Sigma^s_\varepsilon$ starting from $U(\partial K)$ and remaining in $U(\partial K)$,  the map $t \mapsto B(\phi(t))$ is nonincreasing and that  
$\Sigma^s_\varepsilon$ is safe with respect $(X_o,X_u)$,  which completes the proof. 

\item  Under \ref{item:cond4uuu},   we follow the same computations as when $B$ is smooth,  to conclude that,  for some constant 
$\varepsilon > 0$ sufficiently small,  we have 
\begin{align}\label{eqdec}
\langle \zeta ,  -\co \{ F(x+\varepsilon \mathbb{B}) \} + \varepsilon \mathbb{B}  \rangle \subset \mathbb{R}_{<0}
\end{align}
 for all $(\zeta,x) \in \partial_P(-B(x)) \times U( \partial K)$.
Since $- B$ is lower semicontinuous,  
we use \textcolor{blue}{Lemma \ref{lemA8}} to conclude that,  along each solution $\psi$ to 
\begin{align*} 
\Sigma^{s-}_{\varepsilon} : \quad 
\dot{x} \in - \co \{ F(x + \varepsilon \mathbb{B}) \} + \varepsilon \mathbb{B} \qquad x \in \mathbb{R}^n.
\end{align*} 
starting from $ U(\partial K)$ and remaining in $U(\partial K)$, the map $t \mapsto  -B(\psi(t))$ is nonincreasing. Thus, the map 
$t \mapsto  B(\psi(t))$ is nondecreasing.  
Now,  we pick a solution $\phi$  to  $\Sigma^s_\varepsilon$  with 
$\dom \phi = [0, T]$ and $\phi(\dom \phi) \subset  U(\partial K)$. Note that the map $t \mapsto \psi(t) := \phi(T-t)$ is solution to $\Sigma^{s-}_\varepsilon$ satisfying $\psi(\dom \psi) \subset  U(\partial K)$.  Thus, 
$t \mapsto  B(\phi(T-t))$ is nondecreasing,  which implies that 
 $t \mapsto B(\phi(t))$ is non-increasing.  The latter implies,  using \textcolor{blue}{Lemma \ref{lemA8}},  that $\Sigma^s_\varepsilon$ is safe with respect $(X_o,X_u)$,   which completes the proof. 
 \end{itemize}
\end{proof}

\begin{remark}
To guarantee uniform robust safety,  in \cite{RubSafPI},   the inequalities in \ref{item:cond1uuu},  \ref{item:cond2uuu},   
\ref{item:cond3uuu},  and \ref{item:cond4uuu} are used  while replacing  the $\lambda_2$ therein by $0$.  Thus,  Assumption \ref{assusrs} is not needed to characterize uniform robust safety.   However,  as we show in the next example,  for the case when $B$ is continuously differentiable,   the sufficient inequality used in  \cite{RubSafPI} to guarantee uniform robust safety,  which is
\begin{align} \label{eqexp2}
\hspace{-0.8cm} \inf \left\{ \frac{\langle \nabla B (x), -F(x) \rangle}{\lvert \nabla B(x) \rvert} :  x \in  \partial K  \right\} > 0,
\end{align}
 is not strong enough to guarantee uniform strong robust safety. 
\end{remark}

\begin{example} \label{exp2}
Consider the differential inclusion $\Sigma$  with  
$$ F(x) = \begin{bmatrix} F_1(x)  \\  -1+x^2_1x_2 \end{bmatrix}.   $$ 
Consider the initial and unsafe sets 
$X_o = \{x \in \mathbb{R}^2 : x_2 \leq 0\}$  and  $X_u = \mathbb{R}^2 \backslash X_o.  $
Next,  we choose $B(x)= x_2$ as a barrier function candidate with respect to $(X_o,X_u)$.   Note that  in this case $K = X_o$ and  
$ \partial K = \{x \in \mathbb{R}^2 : x_2 = 0\} $.
We start showing that the system is strongly robustly safe with respect to $(X_o, X_u)$.  To do so, we note that
\begin{align*}
\langle \nabla B, F(x) \rangle  =  -1 < 0 \qquad \forall x \in \partial K.
\end{align*}
Now,  we show that system $\Sigma$ is uniformly robustly safe by verifying \eqref{eqexp2}.  Indeed,  we note that 
\begin{align*}
     \frac{\langle \nabla B, - F(x) \rangle}{|\nabla B(x)|} &= 1 - x^2_1x_2 = 1 > 0 \qquad \forall x \in \partial K.
\end{align*}
Hence,  \eqref{eqexp2} holds,  which proves uniform  robust safety. 

Now,  we show that the considered system cannot be uniformly strongly robustly safe by  showing  that,  for each constant $\varepsilon > 0$,   $\Sigma^s_\varepsilon$ admits a solution starting from  $x_o := (1/\sqrt{\varepsilon},  0) \in \partial K$ that enters $X_u$.  To do so,  we start noting that,  for $\mu := [0 ~~ 1]^\top$,  we have     
\begin{align*}
    \langle \nabla B(x_o),  & ~ F(x_o + \varepsilon \mu)  + \varepsilon \mu \rangle  = -1 + \varepsilon + x_{o1}^2 \varepsilon
     \\ & = -1 + \varepsilon  + \left( \frac{1}{\sqrt{\varepsilon}} \right)^2 \varepsilon
    = \varepsilon > 0.
\end{align*}
As a consequence,  using  \textcolor{blue}{\cite[Lemma 4]{draftautomatica}} with $B(x) = -x_2$,  we conclude that 
$ F(x_o + \varepsilon \mu) + \varepsilon \mu \subset D_{\text{int} (X_u)} (x_o), $ 
where $D_K$ is the \textit{Dubovitsky-Miliutin} cone of $K$ defined as 
\begin{align} \label{eq.cone2}
D_K(x) := & \{v \in \mathbb{R}^n : \exists \varepsilon >0: \nonumber \\ &
x + \delta (v + w) \in K~\forall \delta \in (0,\varepsilon],~\forall w \in \varepsilon \mathbb{B}  \}.
\end{align}

Finally,  using \cite[Theorem 4.3.4]{Aubin:1991:VT:120830},   we conclude that the  solution from $x_o$ to the system $ \dot{x} = F(x + \varepsilon \mu)  + \varepsilon \mu $ enters $X_u$.   
\end{example}

\section{Conclusion}

This paper introduced a strong robust-safety notion for continuous-time systems inspired by the scenario of control loops subject to perturbations in both sensing and actuation.  We demonstrated that this notion is stronger than the existing safety and robust-safety notions available in the literature.  Furthermore,  we showed that the existing sufficient conditions for robust safety in \cite{RubSafPI} are strong enough to guarantee strong robust safety under mild regularity assumptions on $F$.  
Following that,  we showed that strong robust safety is equivalent to the existence of a smooth barrier certificate.   Unlike the converse robust safety theorem in \cite{RubSafPII},  we do not require $F$ to be continuous. 
 Furthermore,  we introduced the uniform strong robust-safety notion,  which requires constant strong robust-safety margins and we showed  that the existing sufficient conditions for uniform robust safety in \cite{RubSafPI} are not strong enough to guarantee uniform strong robust safety.  Hence, new sufficient conditions are derived.   
  In future work, it will be interesting to analyze robust safety and strong robust safety for constrained continuous-time systems as well as hybrid systems.

\section*{Acknowledgment}
The authors are grateful to Christophe Prieur for supporting the second author's internship and for the useful feedback on the manuscript.

\section*{Appendix I :   Proof of Propositions}

\subsection{Proof of Proposition  \ref{lemA444}}

The proof follows five steps:

\begin{itemize}
\item \textbf{1st step.} Consider a compact subset $I \subset \mathbb{R}^n$ and a continuous function $\varepsilon : \mathbb{R}^n \rightarrow \mathbb{R}_{>0}$.  Let 
$$ \varepsilon_I := \min\{\varepsilon(x) : x\in I \} > 0.  $$
We will show the existence of $\delta_I > 0$ such that 
\begin{align} \label{eqgraphrev}
  \graph(F) + \delta_I \mathbb{B}  \subset 
  \graph(F_{\varepsilon_{I}}),  
\end{align}
where $ F_{\varepsilon_{I}}(x) := \co \{ F(x + \varepsilon_I\mathbb{B}) \} + \varepsilon_I\mathbb{B}  $,  $\graph(F) \subset \mathbb{R}^{n \times n}$ denotes the graph of $F$,  and \blue{the dimension of $\mathbb{B}$ in \eqref{eqgraphrev} is $n \times n$.}    In other words,   we show the existence $\delta_I >0$ such that,  for each $(x,y) \in (I\times \mathbb{R}^n) \backslash  \graph(F_{\varepsilon_{I}})$, we have   $|(x,y)|_{\graph(F)} > \delta_I$.

To find a contradiction,   we suppose that there exits a sequence $\{(x_i,y_i)\}^{\infty}_{i=0} \subset(I\times \mathbb{R}^n) \backslash  \graph(F_{\varepsilon_{I}})$ 
such that  $\lim_{i \rightarrow \infty} |(x_i,y_i)|_{\graph(F)} = 0$.   Now,  since $I$ is compact,  $F$ is outer semicontinuous and locally bounded, we conclude that $\graph(F)$ restricted to $I$ is compact.  which implies that the sequence $\{(x_i,y_i)\}^{\infty}_{i=0}$ is uniformly bounded.  Hence,  by passing to an appropriate  sub-sequence, we conclude the existence of 
$(x^*, y^*) \in \cl \left( (I\times \mathbb{R}^n) \backslash  \graph(F_{\varepsilon_{I}}) \right)$ with $x^* \in I$ such that $ \lim_{i \rightarrow \infty} (x_i,y_i) = (x^*,y^*)$, which implies,  using the continuity of the distance function,  that 
$$ \lim_{i \rightarrow \infty} |(x_i,y_i)|_{\graph(F)} = |(x^*,y^*)|_{\graph(F)} = 0.  $$
As a result,  we have 
$$ (x^*,y^*) \in \graph(F) \subset \text{int} \left( \graph(F_{\varepsilon_I}) \right),   $$  
which yields to a contradiction.

\item \textbf{2nd step:}  We will show the existence of a sequence  of continuous set-valued maps with convex and compact images,  defined on $I$,  that converges graphically to $F$.

To do so,  we start using  \cite[Theorem 2]{S_M_Aseev} 
to conclude the existence of a sequence $\{ F_k(x) \}^{\infty}_{i=1}$ of continuous set-valued maps with convex and compact images,  defined on $I$,   such that,  for each $x \in I$,  the following properties hold:
\begin{enumerate} [label= (p\arabic*),leftmargin=*]
\item \label{item:p1} $F_{k+1}(x) \subset F_k(x) \qquad \forall k \in \{ 1,2,...\}$.
\item \label{item:p2} $F(x) \subset \text{int} (F_k(x))  \qquad \forall k \in \{ 1,2,... \}$.  
\item \label{item:p3} $F(x) = \lim_{k \to \infty} F_{k}(x) = \bigcap^{\infty}_{k=1} F_{k}(x)$.
\end{enumerate}

Next,   we show that the aforementioned sequence converges graphically to $F$.  To find a contradiction,  we assume that the opposite holds; namely,  there exists 
$\sigma >0$ and a sequence 
$\{(x_k,y_k)\}^{\infty}_{k=1}$ such that $(x_k,y_k)  \in \graph(F_k)$ and
$$ |(x_k,y_k)|_{\graph(F)}>\sigma \qquad \forall k \in \{1,2,...\}.  $$

Since $I$ is compact,  $\{x_k\}^{\infty}_{k=1} \subset I$, 
 and,  under \ref{item:p1},  
$$ y_k \in F_k(x_k) \subset F_{1}(x_k) \subset F_1(I) \quad \forall k \in \{1,2,...\}, $$  
 we conclude that the sequence $\{(x_k, y_k)\}^{\infty}_{k=1}$ is uniformly bounded.   Thus,   by passing to an appropriate subsequence,  we conclude the existence of $(x^*,y^*) \in I \times \mathbb{R}^n$ such that $\lim_{k \rightarrow \infty}  (x_k, y_k)  =  (x^*,y^*)$.  The latter plus the continuity of the distance function allows us to conclude that   
$|(x^*,y^*)|_{\graph(F)}>\sigma,$
which means that  $y^* \notin F(x^*) + \sigma \mathbb{B}$. 

Now,  using \ref{item:p3},  we can choose $k^* \in \{1,2,...\}$ sufficiently large to have 
\begin{align} \label{eqylim+}
F_{k^*} (x^*)  \subset F(x^*) +  (\sigma/4) \mathbb{B}.   
\end{align}
Next,  using \ref{item:p1},  we conclude that 
\begin{align} \label{eqylim-}
y_k  \in  F_{k} (x_k)     \subset  F_{k^*} (x_k) \qquad \forall k \geq k^*.     
\end{align}

Furthermore,  since $F_{k^*}$ is continuous, we conclude the existence of $k^{**} \geq k^*$ such that 
\begin{align} \label{eqylim--}
F_{k^*} (x_k)  \subset  F_{k^*} (x^*)  + (\sigma/4) \mathbb{B} \qquad \forall k \geq k^{**}.  
\end{align}
Moreover,   since $ \lim_{k \rightarrow \infty} y_k = y^*$, we conclude that we can take $k^{**}$ even larger to obtain 
\begin{align} \label{eqylim}
y^* \in y_k +  (\sigma/4) \mathbb{B}   \qquad \forall k \geq k^{**}.    
\end{align} 
Finally,  combining \eqref{eqylim+},  \eqref{eqylim-}, \eqref{eqylim--},  and \eqref{eqylim},   we obtain 
$$ y^* \in    F (x^*)  + (3\sigma/4) \mathbb{B} $$
and the contradiction follows. 

\item \textbf{3rd step:} Using the previous two steps,   we show that there exists a continuous set-valued map $G_I : I \rightrightarrows \mathbb{R}^n$ with convex and compact images such that 
\begin{align*}
    F(x) \subset G_I(x) \subset F_{\varepsilon_I}(x) \qquad \forall x \in I.      
\end{align*}

Indeed,  we showed the existence of a sequence of continuous set-valued maps $\{F_k\}^{\infty}_{k=1}$ that converges graphically to $F$.  
In  particular,  there exists $k_o \in \{1,2,...\} $ such that 
$$  \graph (F_{k_o})  \subset \graph(F) + (\delta_I/2)  \mathbb{B}  \subset \graph (F_{\varepsilon_I}).  $$
Moreover, under \ref{item:p2},  we have  $\graph (F) \subset \graph (F_{k_o})$.    
Hence, the proof is completed by taking $G_I = F_{k_o}$.

\item \textbf{4th step:} 
At this point,  we propose to grid $\mathbb{R}^n$ using a sequence of non-empty compact subsets $\{I_i\}_{i=1}^{N}$,  where $N \in \{1,2,3, \dots ,\infty \}$.  That is,  we assume that $I_i \subset \mathbb{R}^n$ for all $i \in \{1,2, \dots ,N\}$ and $\bigcup^{\infty}_{i = 1} I_i = \mathbb{R}^n$. Furthermore, for each $i \in \{1,2,...,N\}$, there exists a finite set  $\mathcal{N}_i \subset \{1,2,...,N \}$  such that 
\begin{align} \label{eqchhh2}
I_i \cap I_j  & = \emptyset \qquad  \forall j \notin \mathcal{N}_i  
\\
\mbox{int}(I_i \cap I_j)  & = \emptyset \qquad \forall j \in \mathcal{N}_i\backslash \{i\} ,
\end{align}
where $\mbox{int}(\cdot)$ denotes the  interior of $(\cdot)$.

Using the previous steps, we conclude that,   on every subset $I_i$,   there exists a continuous set-valued map $G_{I_{i}} : I_i \rightrightarrows \mathbb{R}^n$ with convex and compact images such that 
\begin{align*}
    F(x) \subset G_{I_{i}}(x) \subset F_{\varepsilon_{I_i}}(x)  \qquad \forall x \in I_i.      
\end{align*}
Now,  we let $G_o : \mathbb{R}^n \rightrightarrows \mathbb{R}^n$ defined as
 \[ G_o(x) :=  \begin{cases} 
G_{I_i}(x)  & \hspace{-0.5cm}  \text{if}~x \in \text{int}(I_i), 
 \\
 \{ y \in G_{I_{j}} (x) : x \in I_j, ~ j \in \mathcal{N}_i\} & \text{otherwise}.
 \end{cases}  \]

 The function $G_o$ is,  by definition,  lower semi-continuous and satisfies 
\begin{align*}
 F(x) \subset  G_o(x) \subset \co\{ F(x + \varepsilon(x) \mathbb{B} ) \} + \varepsilon \mathbb{B} \qquad \forall x \in \mathbb{R}^n. 
\end{align*}
Now, using \cite[Theorem 4]{S_M_Aseev}, we conclude the existence of a continuous set-valued map $G : \mathbb{R}^n \rightrightarrows \mathbb{R}^n$ with convex and compact images such that
\begin{align*}
    F(x) \subset G(x) \subset \co \{ F(x + \varepsilon(x) \mathbb{B}) \} + \varepsilon(x) \mathbb{B} \quad \forall x \in \mathbb{R}^n.
\end{align*} 
\end{itemize}
$\blacksquare$

\subsection{Proof of Proposition  \ref{lem2p}}

We start introducing the continuous function $\beta : \mathbb{R}_{\geq 0}  \times  \mathbb{R}_{\geq 0}   \to \mathbb{R}$ given by
\begin{align*}
\beta(r,\delta) := \max \{ g(y,s) : |y| \leq  r,  ~ 0 \leq s \leq \delta\},  
\end{align*} 
where $g: \mathbb{R}^n \times \mathbb{R}_{\geq 0} \to \mathbb{R}$ is  given by 
\begin{align*}
 g(y,s) := |F(y+s \mathbb{B}) -  F(y)|_H - |F(s \mathbb{B}) - F(0)|_H. 
\end{align*}
It is clear  to see that 
$$ g(x,\delta) \leq \beta(|x|,\delta) \qquad  \forall  (x,\delta)\in \mathbb{R}^n \times \mathbb{R}_{\geq 0}.$$

Furthermore,  we note that the function $\beta$ enjoys the following properties:

\begin{itemize}
\item  $\beta(\cdot,\delta)$ is continuous and nondecreasing for each $\delta \in \mathbb{R}_{\geq 0}$.  
\item  $\beta(r, \cdot)$ is continuous and nondecreasing for each $r \in \mathbb{R}_{\geq 0} $.   
\item  $\beta(0,\delta) = \beta(r, 0) =0$ for all $(r,\delta) \in \mathbb{R}_{\geq 0} \times \mathbb{R}_{\geq 0}$.   
\end{itemize} 

Next,  we will construct a continuous and nondecreasing function 
$\alpha: \mathbb{R}_{\geq 0} \rightarrow \mathbb{R}_{\geq 0}$ such that $\alpha(0) = 0$ and 
$$ \beta(r,\delta) \leq \alpha(r) 
\alpha(\delta) \qquad  \forall 
(r,\delta) \in \mathbb{R}_{\geq 0} \times \mathbb{R}_{\geq 0}.  $$
Indeed,  the latter would imply that
\begin{align*}
g(x,\delta) \leq \alpha(|x|) \alpha(\delta) \qquad \forall (x,\delta) \in \mathbb{R}^n \times \mathbb{R}_{\geq 0}.  
\end{align*}
Hence, for each  $(x,\delta) \in \mathbb{R}^n \times \mathbb{R}_{\geq 0}$,  we have 
\begin{align*}
     |F(x+\delta \mathbb{B}) & - F(x)|_H 
     \leq   \alpha(|x|) \alpha(\delta) + |F(\delta\mathbb{B})- F(0)|_H
    \\ & \leq [\alpha(|x|)+1]  \max\{ \alpha(\delta), |F(\delta\mathbb{B}) - F(0)|_H \}.
\end{align*}
Finally,  we take 
$\lambda_1(\delta) := \max\{\alpha(\delta), |F(\delta\mathbb{B}) - F(0)|_H \}$ and 
$\lambda_2 (x) := \alpha(|x|)+1$. 

The proposed construction of $\alpha$ follows the same steps as in  the proof of \cite[Propositions IV.4 and IV.5]{iisstac}.  The only difference is that,  in the aforementioned reference,  $\beta$ is assumed to be strictly increasing in each of its arguments.   In the following,  we recall such a proof since $\beta$ is only nondecreasing in our case,  and to have an explicit form of the function $\alpha$.   To this end,   we introduce a set of intermediate functions.

\begin{enumerate}
 \item  We define the function $c:\mathbb{R} \times \mathbb{R} \to \mathbb{R}$ given by 
$$c(a,b) := \log \left( \beta \left(e^{a},e^b \right) \right).   $$ 
Note that the function $c$ enjoys the following properties:
\begin{itemize}
\item $c$ is continuous, 
\item $c$ nondecreasing in each of its arguments, 
\item $\lim_{a \rightarrow -\infty} c(a,b) = -\infty$ for all $b \in \mathbb{R}$, 
\item $\lim_{b \rightarrow - \infty} c(a,b) = -\infty$ for all $a \in \mathbb{R}$.
\end{itemize}
 
We will assume,  without loss of generality,   that $c(0,0) > 0$ and we let
$$\bar{a} := \sup \{a \in \mathbb{R} : c(a,0)=0\}. $$
Note that such $\bar{a}$ always exists  since $c(\cdot,0)$ is continuous,  $c(0,0) > 0$, and $c(-\infty,0) = - \infty$.

\item Now,    we introduce a  nonincreasing and continuous function  
$g_o : (-\infty , \bar{a}] \to [0,+\infty)$
such that 
\begin{align*} 
\graph (g_o)  := \{(a,b) \in \mathbb{R} \times \mathbb{R}  :  c(a,b) = -b \}.
\end{align*}
The existence of such a function is shown  in the proof of  
 \cite[Lemma A.1.]{iisstac}.

\item  We introduce the function $g : \mathbb{R} \rightarrow \mathbb{R} $ given by 
\begin{equation*}
\begin{aligned}
 g(a) :=   
 \begin{cases}
-g_o(a)/2  & \text{if}  \;\ a \leq \bar{a} \\
c(a,a) - c(\bar{a},\bar{a}) + a -\bar{a} & \text{if}  \;\ a > \bar{a}. 
\end{cases}
\end{aligned}
\end{equation*}
By construction,  $g$ is continuous,   
nondecreasing,   and we have that  $\lim_{a \rightarrow -\infty} g(a) = - \infty$. 

At this point,  we show that 
$$ \lim_{a \rightarrow \pm \infty}  (c(a,b) - g(a)) = - \infty.  $$
Indeed,  we pick $b \in \mathbb{R}$ and we let  $a \geq \max\{\bar{a},b \}$,  we note that
$$ c(a,b)-g(a) = c(a,b) - c(a,a) - a + \bar{c} \leq \bar{c} -a,$$
where $\bar{c} := \bar{a} + c(\bar{a},\bar{a})$.  
As a result, it follows that 
$$  \lim_{a \rightarrow + \infty}  (c(a,b) - g(a)) = -\infty.   $$

On the other hand,  for any $a \leq \bar{a}$ such that $g_o(a) > b$, we have 
\begin{equation*}
    \begin{aligned}
        c(a,b) & - g(a)  = c(a,b) + \frac{1}{2} g_o(a)
        \\ & \leq c(a,g_o(a)) + g_o(a) -  \frac{1}{2} g_o(a)  
      = - \frac{1}{2} g_o(a),
    \end{aligned}
\end{equation*}
and $\lim_{a \rightarrow -\infty} g_o(a) = +\infty$. 
Hence,  for each $b \in \mathbb{R}$,  we have 
$$ \lim_{a \rightarrow - \infty}  \left(  c(a,b) - g(a) \right) = -\infty.
$$

\item  We introduce the function $h: \mathbb{R} \to \mathbb{R}$ given by
\begin{equation*}
 \begin{aligned}
 h(b) := \sup_{a \in \mathbb{R}} [c(a,b) - g(a)]. 
    \end{aligned}
\end{equation*}
As $a \mapsto c(a,b)-g(a)$ is continuous, and is negative for large $a$,  it follows that $h$ is well defined.   

Since $h(\cdot)$ is the supremum  of a family 
$$ \{c(a,\cdot) - g(a)\}_{a \in \mathbb{R}^n} $$ 
of continuous functions, $h$ is itself continuous, and since,  each member of this family is nondecreasing,  it follows that $h$ is also nondecreasing.  

We prove,  now, that 
$ \lim_{b \rightarrow - \infty} h(b) = -\infty$.  
To do so,   we show that,  for each $v < 0$, we can find $l \leq 0$ such that,  whenever  $b < l$,  then $h(b)< v$. 

 Indeed,   given $v < 0$,   we pick $\rho>0$ so that 
 $$ c(a,0) - g(a) < v \qquad  \forall |a| > \rho.  $$ 
 Next,   we pick an $l  \leq 0$ so that 
 $$ c(\rho, l) - g(-\rho) < v.   $$  
 Such an $l$ exists since $c(\rho, \cdot)$ is unbounded below.    
 
 Consider first the case $|a| \leq \rho$; then 
$$ c(a,b) - g(a) \leq c(\rho,l) - g(-\rho) < v.  
$$
If instead $|a| > \rho$,  then also
$$c(a,b) - g(a) \leq c(a,0) - g(a) < v,$$
using the fact that $b < l \leq 0$.
\end{enumerate}

So,  we have constructed $g$ and $h$ such that 
\begin{equation*}
\begin{aligned}
 c(a, b)  \leq g(a) + h(b).
 \end{aligned}
\end{equation*}

 Hence,   we obtain
$$c(a,b) \leq \max\{g,h\}(a) + \max\{g,h\}(b) \qquad \forall (a,b) \in \mathbb{R} \times \mathbb{R}.$$
Note that, the function $\max\{g,h\}$ is also continuous,   nondecreasing,  and unbounded below.   Hence,  the function $\alpha: \mathbb{R}_{\geq 0} \rightarrow \mathbb{R}_{\geq 0} $ given by  
$$ s \mapsto \alpha(s) :=  \exp(\max\{g,h\}(\ln(s))),  $$  
 is continuous,  non-decreasing,   and satisfies $\alpha(0) = 0$.  
\hfill $\blacksquare$

\section*{Appendix II :   Intermediate Results}

\begin{lemma}  \label{lemA8} 
\blue{Given initial and unsafe sets
$(X_o, X_u) \subset \mathbb{R}^n \times \mathbb{R}^n$,   $\Sigma$ is safe with respect to $(X_o, X_u)$ if there exists  a barrier function candidate $B : \mathbb{R}^n \rightarrow \mathbb{R}$ such that 
\begin{enumerate} 
[label={($\star$)},leftmargin=*]
\item \label{item:star1} 
Along each solution $\phi$ to $\Sigma$ with $\phi(\dom \phi) \subset U(\partial K)$,   the map $t \mapsto B(\phi(t))$ is non-increasing. 
\end{enumerate}
If $B$ is continuous,  we can replace $U(\partial K)$ in \ref{item:star1} by $U(\partial K) \backslash K$.  In turn, 

\begin{enumerate}
\item When $B$ is lower semicontinuous and $F$ is locally Lipschitz,  \ref{item:star1} is satisfied if and only if 
\\
$ \langle \partial_P B(x) ,  \eta  \rangle \subset \mathbb{R}_{\leq 0} \qquad \forall \eta \in F(x), \quad \forall x \in U(\partial K).  $

\item When $B$ is locally Lipschitz,  \ref{item:star1} is satisfied,  while replacing $U(\partial K)$ therein by $U(\partial K) \backslash K$,  if
\\
$
\langle \partial_C B(x), \eta \rangle \subset  \mathbb{R}_{\leq 0}  \quad \forall \eta \in F(x), 
\quad   \forall x \in  U(\partial K) \backslash K.
$
\end{enumerate}}
\end{lemma}

\begin{proof}
\blue{   
To find a contradiction,   we assume the existence of  a solution $\phi$ starting from $x_o \in X_{o}$ that reaches the set $X_{u}$ in finite time,  and we note that the (zero-sublevel) set $K$ satisfies $X_o \subset K$ and $X_u \subset \mathbb{R}^n \backslash K$.  Using continuity of $\phi$,  we conclude the existence of $0 \leq t_1 < t_2$ such that $\phi(t_2) \in  U(\partial K) \backslash K$,  $\phi(t_1) \in K$,  and $\phi([t_1,t_2]) \subset U(\partial K)$.    As a result,  having $B(\phi(t_1)) \leq 0$ and $B(\phi(t_2)) > 0$ contradicts \ref{item:star1}.  

When $B$ is continuous,   we conclude that the (zero-sublevel) set $K$ is closed.   Hence,  there exists $0 \leq t_1 < t_2$ such that $\phi(t_2) \in  U(\partial K) \backslash K$,  $\phi(t_1) \in \partial K$,  and $\phi((t_1,t_2]) \subset U(\partial K) \backslash K$.   Now,  having   $t \mapsto B(\phi(t))$ nonincreasing on $(t_1,t_2]$ allows us to conclude,  using continuity of both $B$ and $\phi$,  that $t \mapsto B(\phi(t))$ is nonincreasing on $[t_1,t_2]$.  But at the same time $B(\phi(t_1)) \leq 0$ and $B(\phi(t_2)) > 0$, which yields to a contradiction.   

The proof of the remaining two statements follows  using \cite[Theorem 6.3]{clarke2008nonsmooth} and 
\cite[Theorem 3]{Sanfelice:monotonicity}, respectively.  }

\end{proof}

\begin{lemma}  \label{lemA333}
Consider a set-valued map
$F: \mathbb{R}^n \rightrightarrows \mathbb{R}^n$ satisfying Assumption \ref{ass1} and such that $F$ is continuous.  
Then,  for each  continuous function $\epsilon : \mathbb{R}^n \to \mathbb{R}_{>0}$, there exists a continuous function $\delta : \mathbb{R}^n \to \mathbb{R}_{>0}$ such that
\begin{align} \label{eqcontic}
F(x+\delta(x)\mathbb{B}) \subset F(x)+ \varepsilon (x)\mathbb{B} \qquad 
\forall x \in \mathbb{R}^n. 
\end{align}
\end{lemma}

\begin{proof}
As in the proof of Proposition \ref{lemA444},   we grid $\mathbb{R}^n$ using a sequence of nonempty compact subsets $\{I_i\}_{i=1}^{N} \subset \mathbb{R}^n$, where $N \in \{1,2,3, \dots ,\infty \}$, such that,  for each $i \in \{1,2,...,N\}$,  there exists 
$\mathcal{N}_i \subset \{1,2,...,N \}$  finite such that 
$ I_i \cap I_j = \emptyset$  for all  $j \notin \mathcal{N}_i$,  and \eqref{eqchhh2} holds. 

Since $F$ is continuous on each $x \in I_i$,  $i \in \{1,2,...,N\}$, then, for each $\varepsilon_i > 0$, there exists $\delta >0$ such that 
\begin{align}
    F(x+\delta \mathbb{B}) \subset F(x)+ \varepsilon_i \mathbb{B}.
\end{align}
Furthermore,  since every continuous $F$ is locally uniformly continuous,  we conclude the existence of $\delta_i > 0$ such that 
\begin{align}
    F(x+\delta_i \mathbb{B}) \subset F(x)+ \varepsilon_i \mathbb{B} \qquad \forall x \in I_i.
\end{align}
Now,  if we let $\varepsilon_i := \min_{x \in I_i} \varepsilon(x)$,  we conclude that 
\begin{align}
F(x+\delta_i \mathbb{B}) \subset F(x)+ \varepsilon(x) \mathbb{B} \qquad \forall x \in 
I_i.  
\end{align}

Next, we introduce the function $\delta_o :\mathbb{R}^n \to \mathbb{R}_{>0}$ given by
$$\delta_{o}(x) :=\text{min}\{\delta_i: i\in \{1,2,\dots , N\} \;\ \text{such that} \;\ x\in I_i\} $$

By definition,  $\delta_{o}$ is lower semi-continuous; hence, using Lemma \ref{lemA3} in Appendix II to conclude the existence of the function $\delta : \mathbb{R}^n \rightarrow \mathbb{R}_{>0}$ continuous and satisfying 
\begin{align}
\delta(x) \leq \delta_{o}(x) \qquad  \forall x \in \mathbb{R}^n.
\end{align}
Hence, \eqref{eqcontic} follows. 
\end{proof}

\begin{lemma} \label{lemA3}
Consider a closed subset $K \subset \mathbb{R}^n$ and a set-valued map $G : K \rightrightarrows \mathbb{R}_{<0}$ that is upper semicontinuous with $G(x)$ nonempty, compact, and convex for all 
$x \in K$. Then, there exists a continuous function $g : K \rightarrow \mathbb{R}_{<0}$ such that  
\begin{align} \label{eqselect1}
0 > g(x) \geq \sup \{\eta: \eta \in G(x)\} \qquad 
\forall x \in K.  
\end{align}
\end{lemma}

\begin{proof}
We start introducing the function 
$f : K \rightarrow \mathbb{R}_{<0}$ given by
$f(x) := \sup \{\eta: \eta \in G(x)\}$. 
\textcolor{blue}{Using \cite[Theorem 1.4.16]{aubin2009set}},  we conclude that $f$ is upper semicontinuous. Finally,  \textcolor{blue}{using  \cite[Theorem 1]{katvetov1951real}},  we conclude that there exists  $g : K \rightarrow \mathbb{R}_{<0}$ continuous such that $g(x) \geq f(x)$ for all $x \in K$.
\end{proof}

\bibliography{biblio}
\bibliographystyle{ieeetr}

\vskip 6pt

\begin{wrapfigure}{l}{18mm}
  \vskip -15pt
  \includegraphics[scale=.2]{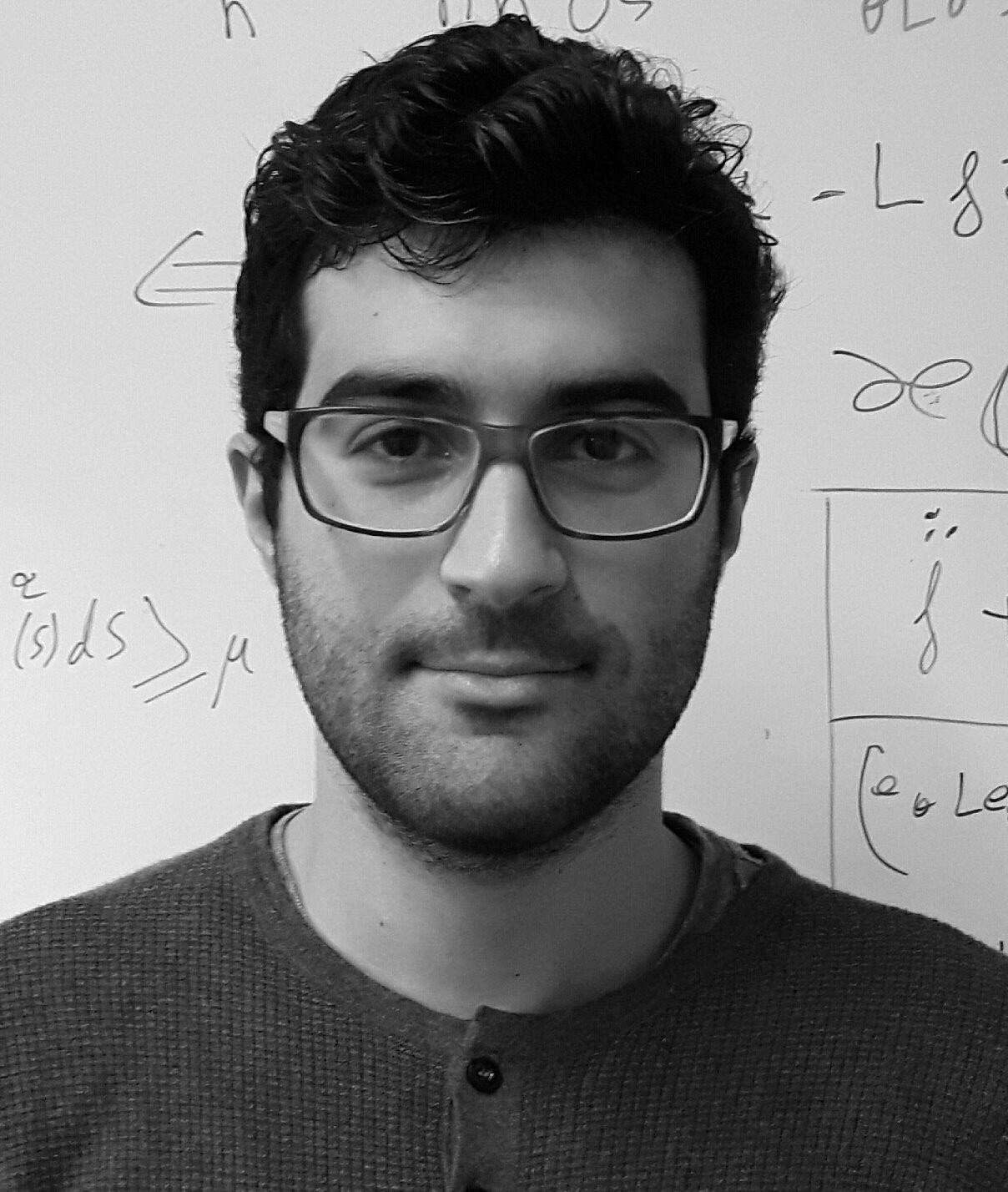} \vskip -1pt
\end{wrapfigure}\par
\footnotesize
\noindent \ \\[-4mm] 
Mohamed Maghenem received his Control-Engineer degree from the Polytechnical School of Algiers, Algeria, in 2013, his  PhD degree on Automatic Control from the University of Paris-Saclay, France, in 2017.  He was a Postdoctoral Fellow at the Electrical and Computer Engineering Department at the University of California at Santa Cruz from 2018 through 2021. M. Maghenem holds a research position at the French National Centre of Scientific Research (CNRS) since January 2021.  His research interests include control systems theory (linear,  non-linear,  and hybrid) to ensure (stability, safety, reachability, synchronisation, and robustness); with applications to mechanical systems,  power systems,  cyber-physical systems,  and some partial differential equations.

\vskip 6pt

\begin{wrapfigure}{l}{22mm}
  \vskip -15pt
 \includegraphics[scale=.1]{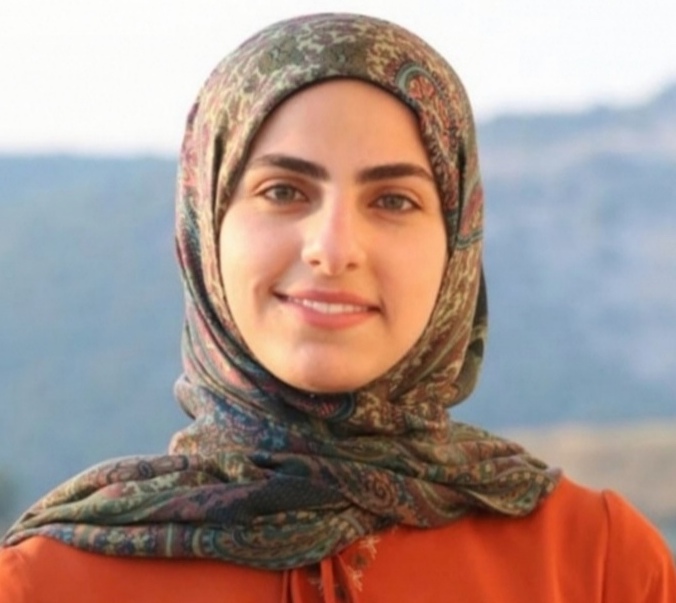}     \vskip -1pt
\end{wrapfigure} 
\par
\footnotesize
\noindent \ \\  [-4mm] 
Diana Karaki received her bachelor and masters I degrees in  mathematics as well as her masters II degree in computer science in 2019, 2020,  and 2021, respectively,  from the Lebanese University,  Lebanon.  She obtained another masters II degree in applied mathematics in 2022 from Limoges University,  France.
She is currently pursuing a research internship at GIPSA-lab, Grenoble, France.


\end{document}